\documentclass{amsart}

\usepackage{amssymb}
\usepackage{latexsym}
\usepackage{amsmath}
\usepackage{euscript}
\usepackage{graphics}

      \def\dC{{\mathbb C}}

   \def\dN{{\mathbb N}}   
      \def\dR{{\mathbb R}}

   \def\dZ{{\mathbb Z}}

\def\cJ{{\mathcal J}}      
   \def\cN{{\mathcal N}}

\def\wh#1{{{\widehat #1} }}

\def\bm\chi{\mbox{\boldmath$\chi$}}

\def\diag{{\rm diag\,}}

\def\e{{\rm e}}
\let\xker=\ker \def\ker{{\xker\,}}
\def\e{\varepsilon}

\def\kk{\kappa}

\newtheorem{thm}{Theorem}[section]
\newtheorem{cor}[thm]{Corollary}

\newtheorem{prop}[thm]{Proposition}
\theoremstyle{definition}
\newtheorem{defn}[thm]{Definition}
\theoremstyle{remark}
\newtheorem{rem}[thm]{Remark}

\theoremstyle{remark}

\numberwithin{equation}{section}

\begin{document}

\title[Darboux transformations and Pad\'e approximation]
{Darboux transformations of  Jacobi matrices \\ and Pad\'e approximation}

\author{Maxim Derevyagin}
\address{
Maxim Derevyagin\\
Department of Mathematics MA 4-5\\
Technische Universit\"at Berlin\\
Strasse des 17. Juni 136\\
D-10623 Berlin\\
Germany}
\email{derevyagin.m@gmail.com}

\author[Vladimir Derkach]{Vladimir Derkach}
\address{Vladimir Derkach\\
Department of Mathematics \\
Donetsk National University \\
Universitetskaya str. 24 \\
83055 Donetsk\\ Ukraine}  \email{derkach.v@gmail.com}
\date{\today}

\begin{abstract}
Let $J$ be a monic Jacobi matrix associated with the
Cauchy transform $F$ of a probability measure.
We construct a pair of the lower and upper triangular
block matrices $L$ and $U$ such that $J=LU$ and the matrix
${\mathfrak J}_C=UL$ is a monic generalized Jacobi matrix associated
with the function ${\mathfrak F}_C(\lambda)=\lambda F(\lambda)+1$.
It turns out that the Christoffel transformation ${\mathfrak J}_C$
of a bounded monic Jacobi matrix $J$ can be unbounded. This
phenomenon is shown to be related to the effect of accumulating at
$\infty$ of the poles of the Pad\'e approximants of the function
${\mathfrak F}_C$ although ${\mathfrak F}_C$ is holomorphic at $\infty$.

The case of the $UL$-factorization of $J$ is considered as well.
\end{abstract}
\maketitle

\section{Introduction}
Let $\sigma$ be a probability measure with an infinite
support contained in $\dR$. Also, assume that all the moments of $\sigma$, i.e.
\[
\int_{\dR}t^jd\sigma(t),\quad j\in\dZ_+,
\]
are finite.
It is well known that the sequence
$\{P_j\}_{j=0}^{\infty}$ of monic polynomials  orthogonal with
respect to $\sigma$ satisfies a three-term recurrence
relation~\cite{Ach61}
\begin{equation}\label{recrelmonp}
\lambda P_j(\lambda) = P_{j+1}(\lambda) + b_jP_j(\lambda) + c_{j-1}P_{j-1}(\lambda),\quad j\in\dZ_+
\end{equation}
with the initial conditions
\[
P_{-1}(\lambda)=0,\quad P_{0}(\lambda) =1,
\]
where $b_j\in\dR$ and $c_j>0$, $j\in\dZ_+$.
Recall that the polynomials $P_j$ are also called the polynomials of the first kind
corresponding to $\sigma$. Besides the polynomials of the first kind,
one can also associate with $\sigma$
polynomials $Q_j$ of the second kind as solutions of~\eqref{recrelmonp}
with the following initial conditions
\[
Q_{-1}(\lambda)=-1,\quad Q_{0}(\lambda) =0.
\]
Clearly, the relation~\eqref{recrelmonp} can be rewritten as follows
\[
Jp(\lambda)=\lambda p(\lambda),
\]
where  $p = (P_0, P_1, P_2,\dots)^{\top}$ and  $J$ is a semi-infinite tridiagonal matrix
of the form
\begin{equation}\label{monJac}
J=\begin{pmatrix}
{b}_{0}   & 1 &       &\\
{c}_{0}   &{b}_1    &{1}&\\
        &{c}_1    &{b}_{2} &\ddots\\
&       &\ddots &\ddots\\
\end{pmatrix}
\end{equation}
The matrix $J$ is said to be {\it the monic Jacobi matrix associated with $\sigma$}.

As is known~\cite{BM04}, a monic Jacobi matrix $J$ admits the unique
$LU$-factorization $J=LU$ with the bidiagonal
matrices $L$ and $U$ having the forms
\begin{equation}\label{LandUdef}
L=\begin{pmatrix}
1   & 0 &       &\\
{l}_{1}   & 1    &0 &\\
        &{l}_2    &1 &\ddots\\
&       &\ddots &\ddots\\
\end{pmatrix},\quad
U=\begin{pmatrix}
{u}_{1}   & 1 &       &\\
{0}   &{u}_2    &{1}&\\
        &0    &{u}_{3} &\ddots\\
&       &\ddots &\ddots\\
\end{pmatrix}.
\end{equation}
if and only if the condition
\begin{equation}\label{ChristCondDef}
P_j(0)\ne 0, \quad j\in\dZ_+
\end{equation}
is satisfied.
{\it The Christoffel transformation} of $J$ is defined as follows
\[
J=LU \rightarrow J_C:=UL.
\]
The matrix $J_{C}$ is the tridiagonal matrix associated to the
measure $td\sigma(t)$~\cite{BM04}.

On the other hand, if $s_{-1}$ is a real number such that
\begin{equation}\label{GerCondDef}
Q_j(0)-s_{-1}P_j(0)\ne 0, \quad j\in\dZ_+,
\end{equation}
then the monic Jacobi matrix $J$ admits  the $UL$-factorization
\[
J=UL
\]
where $U$ and $L$ are defined by~\eqref{LandUdef}. The matrix
\[
J_{G}:=LU
\]
is clearly a tridiagonal matrix and it is
called {\it the Geronimus transformation of $J$ with the parameter $s_{-1}$}.
Notice that  the $LU$-factorization of a monic Jacobi matrix is unique while the $UL$-factorization
depends on the free parameter $s_{-1}$.
Both the Christoffel and Geronimus transformation are also called {\it the Darboux transformations}~\cite{BM04}.
The study of discrete Darboux transformations was originated
in~\cite{Mat79} and has been further developed in~\cite{GH96},~\cite{SpZh95},~\cite{Zh97}.
Also, note that discrete Darboux transformations are applied in bispectral problems~\cite{GH97},
Toda lattices~\cite{MatS79} and numerical linear algebra~\cite{KG}.
We refer the reader to~\cite{BM04} for more details.

In the present paper we extend the Darboux transformations beyond
the conditions~\eqref{ChristCondDef} and~\eqref{GerCondDef}. Namely,
for an arbitrary monic Jacobi matrix $J$ a pair of lower and
upper triangular block matrices $L$ and $U$ is constructed such that
$J=LU$ and the matrix
\[
 {\mathfrak J}_C=UL
\]
is a monic generalized Jacobi matrix in the sense of~\cite{DD09}, i.e. ${\mathfrak J}_C$ is a tridiagonal
block matrix with blocks specified below in Definition~\ref{defGJM}.

Similarly, for any monic Jacobi matrix $J$ and the free parameter $s_{-1}$ the $UL$-factorization
$J={\mathfrak U}{\mathfrak L}$ with the lower and upper triangular block matrices
${\mathfrak L}$ and ${\mathfrak U}$ is found. In this case the Geronimus transform
\[
 {\mathfrak J}_G={\mathfrak L}{\mathfrak U}
\]
is also a tridiagonal block matrix (a monic generalized Jacobi matrix).

Recall that the monic Jacobi matrix $J$ can be also associated with
the Nevanlinna function (see~\cite{Ach61})
\[
 F(\lambda)=\int_{\dR}\frac{d\sigma(t)}{t-\lambda}.
\]
Analogously, one can associate the
generalized Jacobi matrices ${\mathfrak J}_C$ and ${\mathfrak J}_G$
with the functions ${\mathfrak F}_C$ and ${\mathfrak F}_G$, respectively (see~\cite{DD09}).
We show that for the case in question they have the following forms
\[
{\mathfrak F}_C(\lambda)=\lambda F(\lambda)+1,\quad {\mathfrak
F}_G(\lambda)=-\frac{{s}_{-1}}{\lambda}+
\frac{{F}(\lambda)}{\lambda},
\]
where 
$s_{-1}$ is a free parameter.

It turns out that the Christoffel transformation ${\mathfrak J}_C$
of a bounded monic Jacobi matrix can be unbounded, that is, the
entries of ${\mathfrak J}_C$ are not necessarily bounded. In
Section~6 we relate this phenomenon to the effect of accumulating at
$\infty$ of the poles of the Pad\'e approximants to ${\mathfrak
F}_C$. In fact, we show that the $LU$-factorization of $J$ is
bounded in the sense that $L$ and $U$ are bounded if and only if the
diagonal Pad\'e approximants to the corresponding function
${\mathfrak F}_C$ converge to ${\mathfrak F}_C$ locally uniformly
outside of a finite interval of the real line. A similar situation
takes place in the case of the Geronimus transformation.

The paper is organized as follows. In Section~2 we recall some basic facts and definitions.
In particular, the definition of monic generalized Jacobi matrices is given.
The Darboux transformations for arbitrary Jacobi matrices as well as for monic
generalized Jacobi matrices are presented in Sections~3 and~4.
Next section deals with triangular factorizations of arbitrary symmetric Jacobi matrices.
In Section~6, we restrict our consideration to the case of probability measures supported
on $[-1,1]$ and some criteria for the locally uniform convergence, outside of a finite interval, of the diagonal Pad\'e approximants to the functions ${\mathfrak F}_C$ and ${\mathfrak F}_G$ in terms of the $LU$- and $UL$-factorizations are obtained. Finally, in Section~7 we provide the reader with some concrete examples.

\section{Preliminaries}

\subsection{Classes ${\bf D}_{-\infty}^+$ and
${\bf D}^-_{-\infty}$.} Let  ${\mathbf N}$ be the class of all
Nevanlinna functions which map the upper half plane $\dC_+$ into the
upper half plane $\dC_+$. We will say (cf.~\cite{HSWi98}) that a
Nevanlinna function $F\in {\mathbf N}$ belongs to
the class ${\bf N}_{-2n}$ if  for some numbers $
s_0,\dots,s_{2n}\in\dR$ the
following asymptotic expansion holds true
\begin{equation}\label{asymp}
F(\lambda)= 
-\frac{s_{0}}{\lambda}-\frac{s_{1}}{\lambda^2}
-\dots-\frac{s_{2n}}{\lambda^{2n+1}}+o\left(\frac{1}{\lambda^{2n+1}}\right),
\quad\lambda\widehat{\rightarrow }\infty,
\end{equation}
where $\lambda\widehat{\rightarrow }\infty$ means that $\lambda$
tends to $\infty$ nontangentially, i.e. inside the sector
$\varepsilon<\arg \lambda<\pi-\varepsilon$ for some $\varepsilon>0$.
Let us set
\[
{\bf N}_{-\infty}:=\bigcap_{n\ge 0}{\bf N}_{-2n}.
\]

Recall~\cite{Ach61} that $F\in{\bf N}_{-\infty}$ if and only if it
admits the following representation
\[
F(\lambda)=\int_{\dR}\frac{d\sigma(t)}{t-\lambda},
\]
where $\sigma$ is a bounded measure supported on the real line and
$s_j=\int_{\dR}t^jd\sigma(t)$, $j\in\dZ_+$.
It will be sometimes convenient to use the following notation
\begin{equation}\label{eq:asymp}
F(\lambda)\sim-\sum_{j=0}^\infty
\frac{s_j}{{\lambda}^{j+1}},\quad\lambda\widehat{\rightarrow }\infty
\end{equation}
to denote the validity of~\eqref{asymp} for all $n\in\dZ_+$. The
Jacobi matrix $J$ associated with the probability measure
$\sigma/s_0$ will be also called the Jacobi matrix associated with
$F\in {\bf N}_{-\infty}$.

\begin{defn}\label{deffunct+}
Let us say that a function ${\mathfrak F}$ meromorphic in $\dC_+$
belongs to the class ${\bf D}_{-\infty}^+$ if it admits the
asymptotic expansion
\begin{equation}\label{asympF}
{\mathfrak F}(\lambda)\sim-\frac{{\mathfrak
s}_{0}}{\lambda}-\frac{{\mathfrak
s}_{1}}{\lambda^{2}}-\dots-\frac{{\mathfrak s}_{2n}}{\lambda^{2n+1}}
-\dots,\quad\lambda\widehat{\rightarrow }\infty,
\end{equation}
with some ${\mathfrak s}_j\in\dR$ $(j\in\dZ_+)$ and
$\lambda{\mathfrak F}(\lambda)+{\mathfrak s}_{0}\in{\bf
N}_{-\infty}$.
\end{defn}

\begin{defn}\label{deffunct-}
Let us say that a function ${\mathfrak F}$ meromorphic in $\dC_+$
belongs to the class ${\bf D}_{-\infty}^-$ if
there exists a number ${\mathfrak
s}_{-1}\in\dR\setminus\{0\}$ such that
\[
 -\frac{{\mathfrak s}_{-1}}{\lambda}+
\frac{{\mathfrak F}(\lambda)}{\lambda}\in{\bf N}_{-\infty}.
 \]
\end{defn}
In what follows we use the Gothic script for all the notations
associated with the ${\bf D}_{-\infty}^\pm$-functions and the Roman
script for the ${\bf N}_{-\infty}$-functions to avoid confusion.

\subsection{The Schur transform of the ${\bf D}_{-\infty}^\pm$-functions.}

Let ${\mathfrak F}$ be a nonrational function from ${\bf
D}_{-\infty}^\pm$ having an asymptotic expansion of the
form~\eqref{asympF}. Define the set $\cN({\mathfrak s})$ of
normal indices of the sequence ${\mathfrak s}=\{{\mathfrak
s}_{i}\}_{i=0}^\infty$ by
\begin{equation}\label{NormIndFrak}
    \cN({\mathfrak s})=\{{\mathfrak n}_j:\det ({\mathfrak s}_{i+k})_{i,k=0}^{{\mathfrak n}_j-1} \ne 0, \quad j=1,2,\dots\}.
\end{equation}

We will say that the function ${\mathfrak F}\in {\bf D}_{-\infty}^\pm$
is {\it normalized} if the first nontrivial coefficient in its asymptotic expansion~\eqref{asympF} has modulus 1, i.e.
$|{\mathfrak s}_{{\mathfrak n}_1-1}|=1$.

The following statement can be found in~\cite[Theorem~4.8]{DD09}.
\begin{thm}\label{GJM}
Let ${\mathfrak F}$ be a nonrational normalized function of
the class ${\bf D}_{-\infty}^\pm$, let the sequence ${\mathfrak
s}=\{{\mathfrak s}_{j}\}_{j=0}^\infty$ be defined by the asymptotic
expansion~\eqref{asympF}, and let $\cN({\mathfrak s})=\{{\mathfrak
n}_i\}_{i=1}^\infty$ be a set of normal indices of ${\mathfrak
s}$. Then ${\mathfrak F}$ admits the expansion into the following
$P$-fraction
\begin{equation}\label{PfractionG}
-\frac{\epsilon_0}{{\mathfrak p}_0(\lambda)}
\begin{array}{l} \\ -\end{array}
\frac{\epsilon_0\epsilon_1{\mathfrak b}_0^2}{{\mathfrak p}_1(\lambda)}
\begin{array}{ccc} \\ - & \cdots & -\end{array}
\frac{\epsilon_{N-1}\epsilon_N {\mathfrak b}_{N-1}^2}{{\mathfrak
p}_N(\lambda)}\begin{array}{l}\\-\dots\end{array},
\end{equation}
where ${\mathfrak p}_i$ are monic polynomials of degree ${\mathfrak k}_i:={\mathfrak n}_{i+1}-{\mathfrak n}_{i}$
  $(\le 2)$, $\epsilon_i=\pm 1$, ${\mathfrak b}_i>0$, $i\in\dZ_+$, and $\mathfrak{n}_0=0$.
 \end{thm}
For the convenience of the reader we sketch the proof. Actually, the proof is based on
the following step-by-step Schur process~\cite{DD09} (see also~\cite{ADL07},~\cite{De}).
If ${\mathfrak s}_0\ne 0$ then ${\mathfrak n}_1=1$ and the
polynomial ${\mathfrak p}_0$ is defined as follows
\[
 {\mathfrak p}_0(\lambda)=\lambda-\frac{{\mathfrak s}_1}{{\mathfrak s}_0}
=\frac{1 }{{\mathfrak s}_0}
 \begin{vmatrix}
    {\mathfrak s}_0 &{\mathfrak s}_1 \\
     1          & \lambda              \\
  \end{vmatrix}.
\]
If ${\mathfrak s}_0=0$ then ${\mathfrak s}_1\ne 0$, ${\mathfrak n}_1=1$, and
\[
    {\mathfrak p}_0(\lambda)=\frac{1 }{\det({\mathfrak s}_{i+j})_{i,j=0}^{1}}
    \begin{vmatrix}
    0           &  {\mathfrak s}_{1} & {\mathfrak s}_{2} \\
  {\mathfrak s}_{1} & {\mathfrak s}_{2} & {\mathfrak s}_{3} \\
     1          & \lambda   & \lambda^{2} \\
                                                              \end{vmatrix}.
\]
In both cases the Schur transform $\wh{{\mathfrak F}}$ of ${\mathfrak F}\in {\bf D}_{-\infty}^\pm$ is
defined by the equality
\begin{equation}\label{SchTr}
-\frac{1}{{\mathfrak    F}(\lambda)}=\epsilon_0{\mathfrak p}_0(\lambda)+    {\mathfrak
b}_0^2\wh{{\mathfrak F}}(\lambda),
\end{equation}
where $\epsilon_0={\mathfrak s}_{{\mathfrak n}_1-1}$ and ${\mathfrak b}_0$ is chosen in such a way that $\wh{{\mathfrak F}}$ is
normalized. Also, it is easy to see that $\wh{{\mathfrak F}}\in {\bf D}_{-\infty}^\mp$
(see~\cite{DD09} for details).
Further, considering the asymptotic expansion of ${{\mathfrak F}_1}:=\wh{{\mathfrak F}}$
\[
\wh{{\mathfrak F}}(\lambda)\sim -\frac{{\mathfrak s}_{0}^{(1)}}{\lambda}-\frac{{\mathfrak
s}_{1}^{(1)}}{\lambda^{2}}-\dots-\frac{{\mathfrak s}_{2n}^{(1)}}{\lambda^{2n+1}}
-\dots,\quad\lambda\widehat{\rightarrow }\infty,
\]
one can construct the number $\epsilon_1=\pm 1$, the monic polynomial ${\mathfrak p}_1$, and the function ${{\mathfrak F}_2}:=\wh{{\mathfrak F}_1}$.
Continuing this procedure leads us to~\eqref{PfractionG}.

\subsection{Generalized Jacobi matrices associated with the ${\bf D}_{-\infty}^\pm$-functions.}
The $P$-fraction~\eqref{PfractionG} enables us to associate a generalized Jacobi matrix with
the function ${\mathfrak F}\in {\bf D}_{-\infty}^\pm$ (see~\cite{De09,DD,DD07}).

\begin{defn}\label{defGJM}
Let ${\mathfrak F}$ be a nonrational normalized ${\bf D}_{-\infty}^\pm$-function having the $P$-fraction
expansion~\eqref{PfractionG}. Then
{\it a monic generalized Jacobi matrix}  associated with
the function ${\mathfrak F}\in {\bf D}_{-\infty}^\pm$
is the tridiagonal block matrix
\begin{equation}\label{mJacobi}
\mathfrak{J}=\begin{pmatrix}
\mathfrak{B}_{0}   &\mathfrak{D}_{0}&       &\\
\mathfrak{C}_{0}   &\mathfrak{B}_1    &\mathfrak{D}_{1}&\\
        &\mathfrak{C}_1    &\mathfrak{B}_{2} &\ddots\\
&       &\ddots &\ddots\\
\end{pmatrix},
\end{equation}
where the diagonal entries have the following form
\[
{\mathfrak B}_j=\left\{\begin{array}{cc}
             -{\mathfrak p}_0^{(j)}, & \mbox{if }{\mathfrak k}_j=1; \\
             \begin{pmatrix} 0 & 1\\
              -{\mathfrak p}_0^{(j)} & -{\mathfrak p}_1^{(j)}\end{pmatrix} & \mbox{if }{\mathfrak k}_j=2,
           \end{array}\right.
\]
the super-diagonal is built up with $\mathfrak{k}_{j}\times\mathfrak{k}_{j+1}$ matrices
\[
\mathfrak{D}_{j}=
\begin{pmatrix}
0&0\\
1&0\\
\end{pmatrix},\quad
\begin{pmatrix}
0\\
1\\
\end{pmatrix},\quad
\begin{pmatrix}
1&0\\
\end{pmatrix},\quad
(1),
\]
and the sub-diagonal consists of $\mathfrak{k}_{j+1}\times\mathfrak{k}_{j}$ matrices
\[
\mathfrak{C}_{j}=
\begin{pmatrix}
0&0\\
\mathfrak{c}_j&0\\
\end{pmatrix},\quad
\begin{pmatrix}
0\\
\mathfrak{c}_j\\
\end{pmatrix},\quad
\begin{pmatrix}
\mathfrak{c}_j&0\\
\end{pmatrix},\quad
(\mathfrak{c}_j),
\]
here $\mathfrak{c}_j=\epsilon_j\epsilon_{j+1} \mathfrak{b}_j^2$.
\end{defn}
\begin{rem}
Actually, one can associate a monic generalized Jacobi matrix
$\mathfrak{J}$ with an arbitrary nonrational function ${\mathfrak
F}\in {\bf D}_{-\infty}^\pm$. Namely, every function ${\mathfrak
F}\in {\bf D}_{-\infty}^\pm$ can be normalized as follows
${\mathfrak F}_{nor}={\mathfrak F}/|{\mathfrak s}_{{\mathfrak
n}_1-1}|$ and the monic generalized Jacobi matrix associated with
${\mathfrak F}_{nor}$ will be also called associated with
${\mathfrak F}$ (see Remark~2.5).
\end{rem}

Setting $\mathfrak{c}_{-1}=\epsilon_0$, let us define polynomials of
the first kind $\mathfrak{P}_j(\lambda)$, $j\in\dZ_+$, as solutions
$y_j=\mathfrak{P}_j(\lambda)$ of the following system:
\begin{equation}\label{eq10}
\mathfrak{c}_{j-1}y_{j-1}-\mathfrak{p}_j(\lambda)y_{j}+y_{j+1}=0,\,\,\,j\in\dZ_+,
\end{equation}
with the initial conditions
\begin{equation}\label{InConP}
y_{-1}=0,\quad y_0=1.
\end{equation}
Similarly, the polynomials of the second kind $\mathfrak{Q}_j(\lambda)$,
$j\in\dZ_+$, are defined as solutions $u_j=\mathfrak{Q}_j(\lambda)$ of the
system~\eqref{eq10} with the initial conditions
\begin{equation}\label{InConP2}
y_{-1}=-1,\quad y_0=0.
\end{equation}
It follows from~\eqref{eq10} that $\mathfrak{P}_{j}$ is a monic
polynomial of degree $\mathfrak{n}_j$ and $\mathfrak{Q}_j$ is a
polynomial of degree $\mathfrak{n}_j-\mathfrak{k}_0$ with the
leading coefficient $\epsilon_0$. The equations~\eqref{eq10}
coincide with the three-term recurrence relations associated with
$P$-fractions~\cite{Mag1} (see also~\cite[Section 5.2]{JTh}). The
following statement is immediate from~\eqref{eq10}.
\begin{prop}\label{Interlace} {\rm (\cite{DD07}).}
    Polynomials $\mathfrak{P}_{j}$ and $\mathfrak{P}_{j+1}$ ($\mathfrak{Q}_{j}$
and $\mathfrak{Q}_{j+1}$) are coprime.
\end{prop}
Recall (see~\cite{KL79}) that the polynomials ${\mathfrak P}_{j}$
can be found by the formulas
\begin{equation}\label{eq:Pfrak}
    {{\mathfrak P}}_{j}(\lambda)=\frac{1}{d_j} \det\begin{pmatrix}{\mathfrak s}_0 &
{\mathfrak s}_1
& \dots & {\mathfrak s}_{n_j}\\
\hdotsfor4 \\
{\mathfrak s}_{n_j-1} & {\mathfrak s}_{n_j} & \dots & {\mathfrak
s}_{2n_j-1}\\
1 & \lambda & \dots & \lambda^{n_j}\end{pmatrix}\quad
(d_j=\det({\mathfrak s}_{i+k})_{i,k=0}^{n_j-1}),
\end{equation}
and hence
\begin{equation}\label{eq:Pfrak2}
{{\mathfrak P}}_{j}(0)=\frac{(-1)^{\mathfrak{n}_j}}{d_j}\det({\mathfrak
s}_{i+k+1})_{i,k=0}^{n_j-1}, \, {{\mathfrak
Q}}_{j}(0)=\frac{(-1)^{\mathfrak{n}_j}}{d_j} \det\begin{pmatrix}
0&\mathfrak{s}_0&\dots &\mathfrak{s}_{\mathfrak{n}_j-1}\\
\mathfrak{s}_{0}&\mathfrak{s}_1&\dots &\mathfrak{s}_{\mathfrak{n}_j}\\
\hdotsfor{4}\\
\mathfrak{s}_{\mathfrak{n}_j-1}&\mathfrak{s}_{\mathfrak{n}_j+1}&\dots
&\mathfrak{s}_{2\mathfrak{n}_j-1}
\end{pmatrix},
\end{equation}
since ${{\mathfrak Q}}_{j}(\lambda)=
{\mathfrak C}_t\left(\frac{{{\mathfrak P}}_{j}(\lambda)-{{\mathfrak P}}_{j}(t)}{\lambda-t}\right)$,
 where ${\mathfrak C}_t$ is a linear functional such that
 \[
 {\mathfrak C}_t(t^j)={\mathfrak s}_j,\quad j\in\dZ_+.
 \]

Introducing for arbitrary $j\in\dZ_+$ the  shortened matrices
\begin{equation}\label{ShortMat}
\mathfrak{J}_{[0,j]}=
\begin{pmatrix}
\mathfrak{B}_{0}   &\mathfrak{D}_{0}&       &\\
\mathfrak{C}_{0}   & \mathfrak{B}_{1}   & \ddots&\\
        & \ddots   & \ddots &\mathfrak{D}_{j-1}\\
&       &\mathfrak{C}_{j-1} &\mathfrak{B}_j\\
\end{pmatrix},\,
\mathfrak{J}_{[1,j]}=
\begin{pmatrix}
\mathfrak{B}_{1}   &\mathfrak{D}_{1}&       &\\
\mathfrak{C}_{1}   & \mathfrak{B}_{2}   & \ddots&\\
        & \ddots   & \ddots &\mathfrak{D}_{j-1}\\
&       &\mathfrak{C}_{j-1} &\mathfrak{B}_j\\
\end{pmatrix},
\end{equation}
one can obtain the following connection between the polynomials of the
first and second kind $\mathfrak{P}_{j}$, $\mathfrak{Q}_{j}$ and the shortened Jacobi matrices
$\mathfrak{J}_{[0,j]}$ and $\mathfrak{J}_{[1,j]}$
(for the classical case see~\cite[Section 7.1.2]{Be}).
\begin{eqnarray}
\mathfrak{P}_{j}(\lambda)=\det(\lambda-\mathfrak{J}_{[0,j-1]}),\quad
\mathfrak{Q}_{j}(\lambda)=\epsilon_0\det(\lambda-\mathfrak{J}_{[1,j-1]}).\label{polynom2}
\end{eqnarray}

\begin{rem} Let us define an infinite matrix $G$ by the
equality
\begin{equation}\label{Gram}
G=\mbox{diag}(G_{0},\dots,G_{n},\dots),\quad
G_{j}=
\left\{\begin{array}{cc}
             \epsilon_j, & \mbox{if }{\mathfrak k}_j=1; \\
             \epsilon_j\begin{pmatrix} 0 & 1\\
              1 & -{\mathfrak p}_1^{(j)}\end{pmatrix} & \mbox{if }{\mathfrak k}_j=2,
           \end{array}\right.
\end{equation}
and let  $\ell^2_{[0,\infty)}(G)$ be the space of $\ell^2$-vectors
with the inner product
\begin{equation}\label{metric}
\left[x,y\right]=(Gx,y)_{\ell^2_{[0,\infty)}},\quad x,y\in\ell^2_{[0,\infty)}.
\end{equation}
The inner product~\eqref{metric} is indefinite, if either $k_j>1$
for some $j\in\dZ_+$, or at least one $\e_{j}$ is equal to $-1$. The
space $\ell^2_{[0,\infty)}(G)$ is equivalent to a Kre\u{\i}n space
(see~\cite{AI}) if both $G$ and $G^{-1}$ are bounded in
$\ell^2_{[0,\infty)}$.  If $k_j=\e_j=1$ for all $j$ big enough, then
$\ell^2_{[0,\infty)}(G)$ is a Pontryagin space.
\end{rem}

The $m$-function of the shortened matrix $\mathfrak{J}_{[0,j-1]}$ is defined by
\begin{equation}\label{SWfun}
    \mathfrak{m}_{[0,j-1]}(\lambda)=[(\mathfrak{J}^{\top}_{[0,j-1]}-\lambda)^{-1}e_0,e_0],
\end{equation}
where $e_0=(1,0,0,\dots,0)^{\top}\in\dC^{\mathfrak{n}_j}$. Due
to~\eqref{polynom2} it is calculated by
\begin{equation}~\label{mQP}
\mathfrak{m}_{[0,j-1]}(\lambda)=
-\varepsilon_0\frac{\det(\lambda-\mathfrak{J}_{[1,j-1]})}
{\det(\lambda-\mathfrak{J}_{[0,j-1]})}=-\frac{\mathfrak{Q}_{j}(\lambda)}{\mathfrak{P}_{j}(\lambda)}.
\end{equation}
\begin{rem}\label{ResSet}
    Let us emphasize that the polynomials $\mathfrak{P}_{j}$ and $\mathfrak{Q}_{j}$ have no common
    zeros (see~\cite[Proposition 2.7]{DD}) and due to~\eqref{mQP} the set of
    holomorphy of $\mathfrak{m}_{[0,j-1]}$ coincides with the resolvent set of $\mathfrak{J}_{[0,j-1]}$.
\end{rem}

It was shown in~\cite{DD} that  $\mathfrak{m}_{[0,j-1]}$ has the following
property
\begin{equation}\label{Qasympt+1}
\mathfrak{F}(\lambda)-\mathfrak{m}_{[0,j-1]}(\lambda)
=O\left(\frac{1}{\lambda^{2\mathfrak{n}_j+\mathfrak{k}_j}}\right),
\quad\lambda\widehat{\rightarrow }\infty.
\end{equation}
The latter means that
one can reconstruct the sequence $\{\mathfrak{s}_k\}_{k=0}^{\infty}$
by the monic generalized Jacobi matrix $\mathfrak{J}$.
Namely, it follows from~\eqref{Qasympt+1} that for any $k\in\dZ_+$
\begin{equation}\label{momrec}
\mathfrak{s}_k=\left[\Big({\mathfrak
J}_{[0,j-1]}^{\top}\Big)^ke_0,e_0\right],\quad \mbox{ for all
}j\mbox{ such that } \mathfrak{n}_j\ge\frac{k}{2}.
\end{equation}

\section{The Christoffel transformation and its inverse}
From now on we always assume that $F\in{\mathbf N}_{-\infty}$ (or
$\mathfrak{F}\in{\mathbf D}^{\pm}_{-\infty}$) is a nonrational
function. Then there is a monic Jacobi (or generalized Jacobi) matrix
associated with $F$ (or $\mathfrak{F}$, respectively).
\subsection{$LU$-factorizations of Jacobi matrices}
In the following proposition we introduce a block $LU$-factorization of an arbitrary
Jacobi matrix.
\begin{prop}\label{prop31}
Let $J$ be a monic Jacobi matrix associated with $F\in{\mathbf
N}_{-\infty}$ with the asymptotic expansion~\eqref{asymp}, let
$\{{\mathfrak n}_j\}_{j=1}^\infty$ be the set of normal indices of
the sequence ${\mathfrak s}=\{{\mathfrak
s}_{j}\}_{j=0}^\infty:=\{s_{j+1}\}_{j=0}^\infty$ and let ${\mathfrak
k}_j:={\mathfrak n}_{j+1}-{\mathfrak n}_{j}$, $j\in\dZ_+$, where
$\mathfrak{n}_0=0$.
Then $J$
admits the following factorization
\begin{equation}\label{LU}
J=LU,
\end{equation}
where $L$ and $U$ are block lower and upper triangular matrices having the forms
\begin{equation}\label{eq:LU}
    L=\begin{pmatrix}
I_{\mathfrak{k}_{0}} & 0&       &\\
{L}_{1}   & I_{\mathfrak{k}_{1}}   & 0 &\\
        &L_2    & I_{\mathfrak{k}_{2}} &\ddots\\
&       &\ddots &\ddots\\
\end{pmatrix},\quad
U=\begin{pmatrix}
{U}_{0}   &{D}_{0}&       &\\
{0}   &U_{1}    &{D}_{1}&\\
        &0    &{U}_{2} &\ddots\\
&       &\ddots &\ddots\\
\end{pmatrix}
\end{equation}
 in which the sub-diagonal of $L$ consists of
$\mathfrak{k}_{j+1}\times\mathfrak{k}_{j}$ matrices
\begin{equation}\label{eq:L}
{L}_{j+1}=
\begin{pmatrix}
l_{j+1}&0\\
0&0\\
\end{pmatrix},\quad
\begin{pmatrix}
l_{j+1}\\
0\\
\end{pmatrix},\quad
\begin{pmatrix}
l_{j+1}&0\\
\end{pmatrix},\quad
(l_{j+1}),
\end{equation}
the diagonal entries of $U$ are of the form
\begin{equation}\label{eq:U}
{U}_j=\left\{\begin{array}{cc}
             {u}_0^{(j)}, & \mbox{if }{\mathfrak k}_j=1; \\
             \begin{pmatrix} 0 & 1\\
             {u}_0^{(j)}& {u}_1^{(j)}\end{pmatrix} & \mbox{if }{\mathfrak k}_j=2,
           \end{array}\right.
\end{equation}
and the super-diagonal of $U$ is built up with
$\mathfrak{k}_{j}\times\mathfrak{k}_{j+1}$ matrices
\begin{equation}\label{eq:DJ}
    {D}_{j}=
\begin{pmatrix}
0&0\\
1&0\\
\end{pmatrix},\quad
\begin{pmatrix}
0\\
1\\
\end{pmatrix},\quad
\begin{pmatrix}
1&0\\
\end{pmatrix},\quad
(1).
\end{equation}
 Moreover, the factorization~\eqref{LU} is unique and the following
relations hold true
\begin{equation}\label{f_LU_def}
u_0^{(j)}=-\frac{P_{\mathfrak{n}_{j+1}}(0)}{P_{\mathfrak{n}_j}(0)},\quad
u_1^{(j)}=b_{\mathfrak{n}_j+1},\quad j\in\dZ_+.
\end{equation}
\end{prop}
\begin{proof}
It is natural to consider the following relation
\[
LU=
\begin{pmatrix}
U_0 & D_0&       &\\
{L}_{1}U_0   & L_1D_0+U_1   & D_1 &\\
        &L_2U_1    & L_2D_0+U_2 &\ddots\\
&       &\ddots &\ddots\\
\end{pmatrix}=
\begin{pmatrix}
b_0 &1&       &\\
c_{0}&b_{1}    &{1}&\\
        &c_1&{b}_{2} &\ddots\\
&       &\ddots &\ddots\\
\end{pmatrix},
\]
which enables us to find the entries of $L$ and $U$. By the definition, there are two possibilities for
the matrix $U_0$. If $\mathfrak{k}_0=1$ then
\[
u_0^{(0)}=b_0=-\frac{P_1(0)}{P_0(0)}=-\frac{P_{\mathfrak{n}_1}(0)}{P_{\mathfrak{n}_0}(0)}.
\]
In case $\mathfrak{k}_0=2$ we have
\[
\begin{pmatrix}
0&1\\
u_0^{(0)}&u_1^{(0)}\\
\end{pmatrix}=
\begin{pmatrix}
b_0&1\\
c_0&b_1\\
\end{pmatrix},
\]
and, hence, one obtains the equalities
\[
u_0^{(0)}=c_0,\quad u_1^{(0)}=b_1.
\]
Since $\mathfrak{k}_0=2$ then $s_1=0$, $P_1(0)=0$ and~\eqref{eq:1.1}
yields
\[
u_0^{(0)}=c_0=-\frac{P_2(0)}{P_0(0)}=-\frac{P_{\mathfrak{n}_1}(0)}{P_{\mathfrak{n}_0}(0)}.
\]
The equality $b_0=0$ follows from the relation $b_0=\frac{s_1}{s_0}$
(see~\cite{Ach61}).

Now, assume that
\begin{equation}\label{indstep_LU}
u_0^{(j-1)}=-\frac{P_{\mathfrak{n}_{j}}(0)}{P_{\mathfrak{n}_{j-1}}(0)}.
\end{equation}
We will consider  4 cases.

{\it Case 1.} $\mathfrak{k}_j=\mathfrak{k}_{j-1}=2$, $j\in\dN$. In
this case one obtains the following relations
\[
L_jU_{j-1}=\begin{pmatrix}
0&l_{j}\\
0&0\\
\end{pmatrix}=
\begin{pmatrix}
0&c_{\mathfrak{n}_j-1}\\
0&0\\
\end{pmatrix},
\]
\[
L_jD_{j-1}+U_j=
\begin{pmatrix}
0&1\\
u_0^{(j)}&u_1^{(j)}\\
\end{pmatrix}=
\begin{pmatrix}
b_{\mathfrak{n}_j}&1\\
c_{\mathfrak{n}_j}&b_{\mathfrak{n}_j+1}\\
\end{pmatrix},
\]
which immediately imply
\begin{equation}\label{fTh52_1}
l_{j}=c_{\mathfrak{n}_j-1},\quad
u_1^{(j)}=b_{\mathfrak{n}_j+1},\quad u_0^{(j)}=c_{\mathfrak{n}_j}.
\end{equation}
It follows from~$\mathfrak{k}_j=\mathfrak{k}_{j-1}=2$ that
$\det(\mathfrak{s}_{i+l})_{i,l=0}^{\mathfrak{n}_j\pm 1}=0$. Hence
by~\eqref{eq:Pfrak2}
$P_{\mathfrak{n}_j+1}(0)=P_{\mathfrak{n}_j-1}(0)=0$ and
then~\eqref{recrelmonp} yields $b_{n_j}=0$. Moreover,
from~\eqref{recrelmonp} one gets also
\begin{equation}\label{eq:u0j}
u_0^{(j)}=c_{\mathfrak{n}_j}=
-\frac{P_{\mathfrak{n}_{j}+2}(\lambda)+(b_{\mathfrak{n}_{j}+1}-\lambda)P_{\mathfrak{n}_{j}+1}(\lambda)}
{P_{\mathfrak{n}_j}(\lambda)}\Big|_{\lambda=0}=
-\frac{P_{\mathfrak{n}_{j+1}}(0)} {P_{\mathfrak{n}_j}(0)}.
\end{equation}

 {\it Case 2.} $\mathfrak{k}_j=2$, $\mathfrak{k}_{j-1}=1$,
$j\in\dN$. In this case  we have
\[
L_jU_{j-1}=\begin{pmatrix}
l_{j}u_{0}^{(j)}\\
0\\
\end{pmatrix}=
\begin{pmatrix}
c_{\mathfrak{n}_j-1}\\
0\\
\end{pmatrix},
\]
\[
L_jD_{j-1}+U_j=
\begin{pmatrix}
l_j&1\\
u_0^{(j)}&u_1^{(j)}\\
\end{pmatrix}=
\begin{pmatrix}
b_{\mathfrak{n}_j}&1\\
c_{\mathfrak{n}_j}&b_{\mathfrak{n}_j+1}\\
\end{pmatrix}.
\]
Since $P_{\mathfrak{n}_j+1}(0)=0$, one concludes as
in~\eqref{eq:u0j} that
\begin{equation}\label{fTh52_2}
u_1^{(j)}=b_{\mathfrak{n}_j+1},\quad
u_0^{(j)}=c_{\mathfrak{n}_j}=-\frac{P_{\mathfrak{n}_{j+1}}(0)}
{P_{\mathfrak{n}_j}(0)}.
\end{equation}
In order to see that  the relations
\begin{equation}\label{eq_c21}
l_j=b_{\mathfrak{n}_j},\quad l_ju_{0}^{(j-1)}=c_{\mathfrak{n}_j-1}
\end{equation}
hold true simultaneously, let us set $l_j:=b_{\mathfrak{n}_j}$. Thus, according
to the assumption~\eqref{indstep_LU}, the second relation in~\eqref{eq_c21}
can be rewritten as follows
\begin{equation}\label{eq_c21_f}
b_{\mathfrak{n}_j}P_{\mathfrak{n}_j}(0)+c_{\mathfrak{n}_j-1}P_{\mathfrak{n}_j-1}(0)=0,
\end{equation}
and formula~\eqref{eq_c21_f} follows from~\eqref{recrelmonp} and
the fact that $P_{\mathfrak{n}_j+1}(0)=0$.

 {\it Case 3.} $\mathfrak{k}_j=1$, $\mathfrak{k}_{j-1}=2$, $j\in\dN$. The identities
$\mathfrak{k}_j=1$ and $\mathfrak{k}_{j-1}=2$ yield
\[
L_jU_{j-1}=\begin{pmatrix}
0&l_{j}\\
\end{pmatrix}=
\begin{pmatrix}
0&c_{\mathfrak{n}_j-1}\\
\end{pmatrix},
\quad L_jD_{j-1}+U_j=u_0^{(j)}=b_{\mathfrak{n}_j}.
\]
Consequently, by taking into account the equalities $P_{\mathfrak{n}_j-1}(0)=0$ and
$\mathfrak{n}_{j+1}=\mathfrak{n}_j+1$ we get that
\begin{equation}\label{fTh52_3}
l_{j}=c_{\mathfrak{n}_j-1},\quad
u_0^{(j)}=b_{\mathfrak{n}_j}=-\frac{P_{\mathfrak{n}_{j+1}}(0)}
{P_{\mathfrak{n}_j}(0)}.
\end{equation}

 {\it Case 4.} $\mathfrak{k}_j=1$, $\mathfrak{k}_{j-1}=1$, $j\in\dN$. This case reduces
to the following relations
\begin{equation}\label{eq_c11}
L_jU_{j-1}=l_{j}u_{0}^{(j)}= c_{\mathfrak{n}_j-1}, \quad
L_jD_{j-1}+U_j= l_j+u_0^{(j)}= b_{\mathfrak{n}_j}.
\end{equation}
Eliminating $l_j$ from the first equation in~\eqref{eq_c11} and substituting it to
the second one, we get
\begin{equation}\label{fTh52_4}
u_0^{(j)}=
\frac{b_{\mathfrak{n}_{j}}P_{\mathfrak{n}_{j}}(0)+c_{\mathfrak{n}_{j}-1}P_{\mathfrak{n}_{j}-1}(0)}
{P_{\mathfrak{n}_j}(0)}= -\frac{P_{\mathfrak{n}_{j+1}}(0)}
{P_{\mathfrak{n}_j}(0)}.
\end{equation}
\end{proof}
\begin{thm}\label{Chris}
Let $J$ be a monic Jacobi matrix associated with $F\in{\mathbf N}_{-\infty}$ and let
$J=LU$ be its $LU$ factorization. Then the matrix $\mathfrak{J}=UL$ is the monic generalized
Jacobi matrix associated with $\lambda F(\lambda)+s_0\in{\bf D}_{-\infty}^-$.
\end{thm}
\begin{proof}
The fact that $\mathfrak{J}=UL$ is a monic generalized
Jacobi matrix can be easily verified by straightforward calculations.

We can assume that the function $F$ is normalized and let $F$ have
the asymptotic expansion~\eqref{eq:asymp}. Then the function
$\lambda F(\lambda)+1$ has the asymptotic expansion
\[
\lambda
F(\lambda)+1\sim-\frac{s_1}{\lambda}-\frac{s_2}{\lambda^2}-\dots\quad(\lambda\wh\to\infty)
\]
and is not necessarily normalized. Since according
to~\eqref{f_LU_def} we have
\[
\epsilon_0u_{0}^{(0)}=-\epsilon_0\frac{P_{\mathfrak{n}_{1}}(0)}{P_{\mathfrak{n}_0}(0)}=
\left\{\begin{array}{cc}
           \epsilon_0  s_1, & \mbox{if }{\mathfrak k}_0=1; \\
            \epsilon_0 s_2, & \mbox{if }{\mathfrak k}_0=2
           \end{array}\right.=
  \left\{\begin{array}{cc}
             |s_1|, & \mbox{if }{\mathfrak k}_0=1; \\
             |s_2|, & \mbox{if }{\mathfrak k}_0=2
           \end{array}\right.
\]
the function ${\mathfrak
F}_{nor}=(F(\lambda)+1)/{\epsilon_0u_{0}^{(0)}}$ is normalized.

 Further, observe that
\begin{equation}\label{inCT}
L_{[0,j]}^{\top}e_0=e_0,\quad
G_{[0,j]}e_0=\frac{\epsilon_0}{u_0^{(0)}}U_{[0,j]}e_0,\quad
j\in\dZ_+,
\end{equation}
where the shortened matrices $L_{[0,j]}$, $U_{[0,j]}$, and $G_{[0,j]}$ are defined analogously to~\eqref{ShortMat}.

Clearly, for $j$ big enough
\[
s_n=\left(e_0,J_{[0,{\mathfrak n}_{j+1}-1]}^ne_0\right), \quad n\in\dZ_+,
\]
where $J_{[0,{\mathfrak n}_{j+1}-1]}$ is the leading principal submatrix of
$J$ constructed from the first ${\mathfrak n}_{j+1}$ rows and
columns of $J$. Thus, we have
\[
\begin{split}
s_n=\left(e_0,J_{[0,{\mathfrak n}_{j+1}-1]}^ne_0\right)=
\Big(e_0,\underbrace{L_{[0,j]}U_{[0,j]}\dots L_{[0,j]}U_{[0,j]}}_{n\quad\text{times}} e_0\Big)=\\
=\Big(L^{\top}e_0,\underbrace{U_{[0,j]}L_{[0,j]}\dots U_{[0,j]}L_{[0,j]}}_{n-1\quad\text{times}} U_{[0,j]}e_0\Big),
\quad n\in\dN.
\end{split}
\]
Further, using~\eqref{inCT} gives the equality
\[
s_n=\epsilon_0u_{0}^{(0)}\left(e_0,(U_{[0,j]}L_{[0,j]})^{n-1}Ge_0\right).
\]
Hence for sufficiently large $j$ one calculates the moments
\begin{equation}\label{norm_def_sec}
\left[\left({\mathfrak{J}^{\top}_{[0,j]}}\right)^{n-1}e_0,e_0\right]=
\frac{s_n}{\epsilon_0u_{0}^{(0)}},
\end{equation}
which coincide with the coefficients in the asymptotic expansion of
${\mathfrak F}_{nor}$. In view of~\eqref{momrec} this means that the
generalized Jacobi matrix ${\mathfrak{J}}$ is associated with the
function $\lambda F(\lambda)+s_0\in{\bf D}_{-\infty}^-$.
\end{proof}

The transform $J=LU\mapsto \mathfrak{J}=UL$ is called the {\it
Christoffel transform} of the Jacobi matrix $J$.

\begin{cor}
Let $J$ be a monic Jacobi matrix associated with $F\in{\mathbf
N}_{-\infty}$ with the asymptotic expansion~\eqref{asymp} 
Assume that
$P_j(0)\ne 0$ for all $j\in\dN$. Then ${\mathfrak n}_j=j$ for
$j\in\dN$ and
\begin{enumerate}
    \item [(i)]  The matrix $J$
admits the $LU$-factorization~\eqref{LU}, where
$L$ and $U$ are lower and upper triangular matrices
having the forms \eqref{eq:LU} with
\begin{equation}\label{f_LU_BM}
D_j=1,\quad U_j=-\frac{P_{j+1}(0)}{P_{j}(0)},\quad
L_{j+1}=b_{{j+1}}-U_{j+1},\quad j\in\dZ_+.
\end{equation}
\item[(ii)]
The matrix $\mathfrak{J}=UL$ is the monic tridiagonal generalized
Jacobi matrix associated with $\lambda F(\lambda)+s_0\in{\bf
D}_{-\infty}^-$.
\end{enumerate}
\end{cor}

The statement of this corollary is contained in~\cite{BM04}, where,
in fact, more general monic three-diagonal  matrices of the
form~\eqref{monJac} were considered (their matrices are associated with a 
class of quasi-definite linear functionals which includes finite measures 
on the real line as a subclass).

\subsection{$UL$-factorizations of generalized Jacobi matrices}
Here we present a block $UL$-factorization of monic generalized Jacobi matrices.
\begin{prop}\label{prop:3.3}
Let $\mathfrak{J}$ be a monic generalized Jacobi matrix associated
with $\mathfrak{F}\in{\bf D}_{-\infty}^-$, let $\{{\mathfrak
n}_j\}_{j=1}^\infty$ be the set of normal indices of the sequence
${\mathfrak s}=\{{\mathfrak s}_{j}\}_{j=0}^\infty$ defined by the
asymptotic expansion~\eqref{asympF} of the function $\mathfrak{F}$,
${\mathfrak k}_j:={\mathfrak n}_{j+1}-{\mathfrak n}_{j}$,
$j\in\dZ_+$, where $\mathfrak{n}_0=0$,  let ${\mathfrak
s}_{-1}\in\dR$ be the same as in Definition~\ref{deffunct-}, let
${\mathfrak P}_j(\lambda)$ and ${\mathfrak Q}_j(\lambda)$ be
polynomials of the first and the second kind associated with the
sequence ${\mathfrak s}=\{{\mathfrak s}_{j}\}_{j=0}^\infty$, and let
the polynomials $\wh{{\mathfrak P}}_j$ be defined by
\[
\wh{{\mathfrak P}}_j(\lambda)=
\mathfrak{P}_j(\lambda)-\frac{1}{\mathfrak{s}_{-1}}\mathfrak{Q}_j(\lambda).
\]
Then:
\begin{enumerate}
    \item[(i)] $\wh{{\mathfrak P}}_j(0)\ne 0$ for $j\in\dZ_+$;
    \item[(ii)] the GJM $\mathfrak{J}$ admits the following factorization
\begin{equation}\label{ind_UL}
\mathfrak{J}=UL 
\end{equation}
where $U$ and $L$ 
are block lower and upper triangular matrices having the
form~\eqref{eq:LU}--\eqref{eq:DJ}. Moreover, the following relations
hold true
\begin{equation}\label{f_UL_ind}
{u}_1^{(j)}=-\mathfrak{p}_1^{(j)},\quad
{l}_{j+1}=-\frac{\wh{{\mathfrak P}}_{j+1}(0)}{\wh{{\mathfrak
P}}_{j}(0)},\quad {u}_0^{(0)}=\frac{1}{\mathfrak{s}_{-1}},\quad
{u}_0^{(j+1)}=\frac{\mathfrak{c}_{j}} {{l}_{j+1}}.
\end{equation}
\end{enumerate}

\end{prop}
\begin{proof}
(i) Notice first that the sequence
$\{s_j\}_{j=0}^\infty=\{\mathfrak{s}_{j-1}\}_{j=0}^\infty$ satisfies
the asymptotic expansion~\eqref{asympF} of the function
$F(\lambda)=\frac{\mathfrak{F}(\lambda)-\mathfrak{s}_{-1}}{\lambda}\in{\mathbf
N}_{-\infty}$. Then
\[
\mathfrak{d}_k:=\det
\begin{pmatrix}
\mathfrak{s}_{-1}&\mathfrak{s}_0&\dots &\mathfrak{s}_{k}\\
\mathfrak{s}_{0}&\mathfrak{s}_1&\dots &\mathfrak{s}_{k+1}\\
\hdotsfor{4}\\
\mathfrak{s}_{k}&\mathfrak{s}_{k+1}&\dots &\mathfrak{s}_{2k+1}
\end{pmatrix}\ne 0,\quad k\in\dZ_+,
\]
by Definition~\ref{deffunct-} and the property of Nevanlinna
functions. In particular, one gets from~\eqref{eq:Pfrak2} the
relation
\[
    \begin{split}
\mathfrak{d}_{\mathfrak{n}_j-1} &=\mathfrak{s}_{-1}
\det\begin{pmatrix}
\mathfrak{s}_{1}&\dots &\mathfrak{s}_{\mathfrak{n}_j}\\
\hdotsfor{3}\\
\mathfrak{s}_{\mathfrak{n}_j}&\dots
&\mathfrak{s}_{2\mathfrak{n}_j-1}
\end{pmatrix}+
\det\begin{pmatrix}
0&\mathfrak{s}_0&\dots &\mathfrak{s}_{\mathfrak{n}_j-1}\\
\mathfrak{s}_{0}&\mathfrak{s}_1&\dots &\mathfrak{s}_{\mathfrak{n}_j}\\
\hdotsfor{4}\\
\mathfrak{s}_{\mathfrak{n}_j-1}&\mathfrak{s}_{\mathfrak{n}_j+1}&\dots
&\mathfrak{s}_{2\mathfrak{n}_j-1}
\end{pmatrix}\\
&=(-1)^{j+1}
\left(\mathfrak{s}_{-1}\mathfrak{P}_j(0)-\mathfrak{Q}_j(0)\right)
\det(\mathfrak{s}_{i+k})_{i,k=0}^{\mathfrak{n}_j-1}\ne 0,
\end{split},
\]
which shows that $\wh{{\mathfrak P}}_j(0)\ne 0$ for $j\in\dZ_+$.

 (ii) Further, the equality $\mathfrak{J}={U}{L}$ yields
\begin{equation}\label{helpUL1}
{u}_1^{(j)}=-\mathfrak{p}_1^{(j)},\quad
{u}_0^{(j)}+{l}_{j+1}=-\mathfrak{p}_0^{(j)},\quad
{u}_0^{(j+1)}{l}_{j+1}=\mathfrak{c}_j,\quad j\in\dZ_+.
\end{equation}
Setting ${u}_0^{(0)}:={1}/{\mathfrak{s}_{-1}}$, we see that
\[
{l}_{1}=-\mathfrak{p}_0^{(0)}-\frac{1}{\mathfrak{s}_{-1}}=
-\frac{\wh{{\mathfrak P}}_{1}(0)}{\wh{{\mathfrak P}}_{0}(0)}.
\]

Next, the second formula in~\eqref{f_UL_ind} follows  by induction
with the help of the second and third equations in~\eqref{helpUL1},
and the fact that the polynomials $\wh{{\mathfrak P}}_j$
satisfy~\eqref{eq10}.
\end{proof}

\begin{rem}
It should be noted that, actually, both the matrices in the $UL$-decom\-position~\eqref{ind_UL}
depend on $\mathfrak{s}_{-1}$, that is, $U=U(\mathfrak{s}_{-1})$ and $L=L(\mathfrak{s}_{-1})$.
\end{rem}

\begin{thm}
Let $\mathfrak{J}$ be a monic generalized Jacobi matrix associated
with $\mathfrak{F}\in{\bf D}_{-\infty}^-$, let ${\mathfrak
s}_{-1}\in\dR$ be the same as in Definition~\ref{deffunct-}, and let
$\mathfrak{J}=UL$ be its $UL$ factorization of the
form~\eqref{ind_UL}, ~\eqref{eq:LU}--\eqref{eq:DJ}. Then the matrix
$J=LU$ is the monic Jacobi matrix associated with
$\frac{\mathfrak{F}(\lambda)}{\lambda}-\frac{\mathfrak{s}_{-1}}{\lambda}\in{\mathbf
N}_{-\infty}$.
\end{thm}
\begin{proof}
The fact that $J$ is a classical
Jacobi matrix can be easily verified by straightforward calculations.
The rest of the proof can be done by reversing the reasoning given
in the proof of Theorem~\ref{Chris}.
\end{proof}

\section{The Geronimus transformation and its inverse}

\subsection{$LU$-factorizations of generalized Jacobi matrices}
\begin{prop}
Let $\mathfrak{J}$ be a monic generalized Jacobi matrix associated
with $\mathfrak{F}\in{\bf D}_{-\infty}^+$ and let ${\mathfrak
k}_j:={\mathfrak n}_{j+1}-{\mathfrak n}_{j}$,  $j\in\dZ_+$, where
$\mathfrak{n}_0=0$ and $\{{\mathfrak n}_j\}_{j=1}^\infty$ is the set
of normal indices of the sequence ${\mathfrak s}=\{{\mathfrak
s}_{j}\}_{j=0}^\infty$ defined by~\eqref{asympF} and let
${{\mathfrak P}}_j(\lambda)$ be polynomials of the first kind
associated with the sequence ${\mathfrak s}=\{{\mathfrak
s}_{j}\}_{j=0}^\infty$. Then ${{\mathfrak P}}_j(0)\ne 0$ for all
$j\in\dZ_+$ and the GJM $\mathfrak{J}$ admits the following
factorization
\begin{equation}\label{ind_LU}
\mathfrak{J}=\mathfrak{L}\mathfrak{U},
\end{equation}
where $\mathfrak{L}$ and $\mathfrak{U}$ are block lower and upper
triangular matrices having the forms
\begin{equation}\label{eq:ind_LU}
\mathfrak{L}=\begin{pmatrix}
\mathfrak{E}_{0}& 0&       &\\
\mathfrak{L}_{1}   & \mathfrak{E}_{1}   & 0 &\\
        &\mathfrak{L}_2    & \mathfrak{E}_{2} &\ddots\\
&       &\ddots &\ddots\\
\end{pmatrix},\quad
U=\begin{pmatrix}
\mathfrak{U}_{0}   &\mathfrak{D}_{0}&       &\\
{0}   &\mathfrak{U}_{1}    &\mathfrak{D}_{1}&\\
        &0    &\mathfrak{U}_{2} &\ddots\\
&       &\ddots &\ddots\\
\end{pmatrix}
\end{equation}
in which the sub-diagonal of $\mathfrak{L}$ consists of
$\mathfrak{k}_{j}\times\mathfrak{k}_{j-1}$ matrices
\begin{equation}\label{eq:ind_L}
\mathfrak{L}_{j}=
\begin{pmatrix}
0&0\\
0&\mathfrak{l}_{j}\\
\end{pmatrix},\quad
\begin{pmatrix}
0\\
\mathfrak{l}_{j}\\
\end{pmatrix},\quad
\begin{pmatrix}
0&\mathfrak{l}_{j}\\
\end{pmatrix},\quad
(\mathfrak{l}_{j}),
\end{equation}
$\mathfrak{D}_{j}$ are of the form~\eqref{eq:DJ} and
$\mathfrak{U}_j$, $\mathfrak{E}_j$ are
$\mathfrak{k}_{j}\times\mathfrak{k}_{j}$ matrices
\begin{equation}\label{eq:ind_U}
\mathfrak{U}_j=\left\{\begin{array}{cc}
             \mathfrak{u}_0^{(j)}, & \mbox{if }{\mathfrak k}_j=1; \\
             \begin{pmatrix} 0 & 1\\
              \mathfrak{u}_{j}& 0 \end{pmatrix} & \mbox{if }{\mathfrak k}_j=2.
           \end{array}\right.,\quad
\mathfrak{E}_j=\left\{\begin{array}{cc}
             1, & \mbox{if }{\mathfrak k}_j=1; \\
             \begin{pmatrix} 1 & 0\\
              \mathfrak{e}_{j}& 1 \end{pmatrix} & \mbox{if }{\mathfrak k}_j=2.
           \end{array}\right..
\end{equation}
Moreover, the following relations hold true
\begin{equation}\label{eq:LU_ind}
\mathfrak{u}_{j}=-\frac{{{\mathfrak P}}_{j+1}(0)}{{{\mathfrak
P}}_{j}(0)},\quad
\mathfrak{l}_{j+1}=\frac{\mathfrak{c}_{j}}{\mathfrak{u}_{j}},\quad
\mathfrak{e}_j=-\mathfrak{p}_1^{(j)}, \quad j\in\dZ_+.
\end{equation}
\end{prop}
\begin{proof}
First, notice that it follows from formula~\eqref{eq:Pfrak}
that
\[
{{\mathfrak P}}_{j}(0)=\frac{\det({\mathfrak
s}_{i+k+1})_{i,k=0}^{n_j-1}}{\det({\mathfrak
s}_{i+k})_{i,k=0}^{n_j-1} }\ne 0,
\]
since ${\mathfrak F}\in{\bf D}_{-\infty}^+$ and $s_i:={\mathfrak
s}_{i+1}$ $(i\in\dZ_+)$ are moments of the Nevanlinna function
$\lambda{\mathfrak F}(\lambda)+{\mathfrak s}_0\in{\bf
N}_{-\infty}$.

Consider the product $\mathfrak{L}\mathfrak{U}$ of the matrices
$\mathfrak{L}$ and $\mathfrak{U}$
\[
\mathfrak{L}\mathfrak{U}=
\begin{pmatrix}
\mathfrak{E}_0 \mathfrak{U}_0&  \mathfrak{E}_0\mathfrak{D}_0  & &\\
& & & \\
\mathfrak{L}_{1}\mathfrak{U}_0   &
\mathfrak{L}_1\mathfrak{D}_0+\mathfrak{E}_1\mathfrak{U}_1
& \mathfrak{E}_1\mathfrak{D}_1 &\\
&\mathfrak{L}_2\mathfrak{U}_1  & \mathfrak{L}_2\mathfrak{D}_1+\mathfrak{E}_2\mathfrak{U}_2 &\ddots\\
& & & \\
&       &\ddots &\ddots\\
\end{pmatrix}
=\begin{pmatrix}
\mathfrak{B}_{0}   &\mathfrak{D}_{0}&       &\\
& & & \\
\mathfrak{C}_{0}   &\mathfrak{B}_1    &\mathfrak{D}_{1}&\\
        &\mathfrak{C}_1    &\mathfrak{B}_{2} &\ddots\\
        & & & \\
&       &\ddots &\ddots\\
\end{pmatrix},
\]
Comparing it with the matrix ${\mathfrak J}$ in~\eqref{mJacobi} one
finds the  entries of ${\mathfrak L}$ and ${\mathfrak U}$.

Next, we will prove formula~\eqref{eq:LU_ind}.

If $\mathfrak{k}_0=1$ then the equality ${\mathfrak E}_0{\mathfrak
U}_0={\mathfrak B}_0$ yields
\begin{equation}\label{eq:4.5}
    {\mathfrak u}_0=-{\mathfrak p}_0^{(0)}=-\frac{{\mathfrak
P}_1(0)}{{\mathfrak P}_0(0)}.
\end{equation}
 In the case $\mathfrak{k}_0=2$ we have
\[
{\mathfrak E}_0{\mathfrak U}_0=\begin{pmatrix}
0&1\\
{\mathfrak u}_0& {\mathfrak e}_0\
\end{pmatrix}=
\begin{pmatrix} 0 & 1\\
              -{\mathfrak p}_0^{(j)} & -{\mathfrak p}_1^{(j)}\end{pmatrix} \]
and, again, we obtain the equalities
\[
{\mathfrak u}_0=-{\mathfrak p}_0^{(0)}=-\frac{{\mathfrak
P}_1(0)}{{\mathfrak P}_0(0)},\quad {\mathfrak e}_0=-{\mathfrak
p}_1^{(0)}.
\]

Now, assume that
\begin{equation}
{\mathfrak u}_{j-1}=\frac{{\mathfrak P}_{j}(0)}{{\mathfrak
P}_{j-1}(0)}, \quad j\in\dN.
\end{equation}
We will analyze four cases.

{\it Case 1.} $\mathfrak{k}_j=\mathfrak{k}_{j-1}=2$, $j\in\dN$. In
this case one gets the following relations
\begin{equation}\label{eq:LUj}
{\mathfrak L}_j{\mathfrak U}_{j-1}=\begin{pmatrix}
0&0\\
{\mathfrak l}_j{\mathfrak u}_{j-1}&0\\
\end{pmatrix}=
\begin{pmatrix}
0& 0\\
{\mathfrak c}_{j-1}&0\\
\end{pmatrix},
\end{equation}
\begin{equation}\label{eq:LDj}
 \mathfrak{L}_j\mathfrak{D}_{j-1}+\mathfrak{E}_j\mathfrak{U}_j=
\begin{pmatrix}
0&1\\
{\mathfrak l}_{j}+{\mathfrak u}_{j}&{\mathfrak e}_{j}\\
\end{pmatrix}=
\begin{pmatrix} 0 & 1\\
              -{\mathfrak p}_0^{(j)} & -{\mathfrak p}_1^{(j)}\end{pmatrix}
\end{equation}
Hence one obtains
\begin{equation}\label{eq:lej}
{\mathfrak
l}_{j}=\frac{\mathfrak{c}_{j-1}}{\mathfrak{u}_{j-1}},\quad
{\mathfrak e}_{j}=-{\mathfrak p}_1^{(j)},
\end{equation}
\begin{equation}\label{eq:uj}
{\mathfrak u}_{j}=-({\mathfrak p}_0^{(j)}+{\mathfrak
l}_{j})=-\frac{{\mathfrak p}_{j}(0){\mathfrak P}_{j}(0)-{\mathfrak
c}_{j-1}{\mathfrak P}_{j-1}(0)}{{\mathfrak
P}_{j}(0)}=-\frac{{\mathfrak P}_{j+1}(0)}{{\mathfrak P}_{j}(0)}.
\end{equation}

{\it Case 2.} Let $\mathfrak{k}_j=2$, $\mathfrak{k}_{j-1}=1$,
$j\in\dN$. Then~\eqref{eq:LDj} holds true and~\eqref{eq:LUj} takes
the form
\[
{\mathfrak L}_j{\mathfrak U}_{j-1}=\begin{pmatrix}
0\\
{\mathfrak l}_j{\mathfrak u}_{j-1}\\
\end{pmatrix}=
\begin{pmatrix}
0\\
{\mathfrak c}_{j-1}\\
\end{pmatrix},
\]
Hence one obtains~\eqref{eq:lej} and~\eqref{eq:uj}.

{\it Case 3.} Let $\mathfrak{k}_j=1$, $\mathfrak{k}_{j-1}=2$,
$j\in\dN$. Then
\begin{equation}\label{eq:LUj21}
{\mathfrak L}_j{\mathfrak U}_{j-1}=\begin{pmatrix}
{\mathfrak l}_j{\mathfrak u}_{j-1}&0\\
\end{pmatrix}=
\begin{pmatrix}
{\mathfrak c}_{j-1}&0\\
\end{pmatrix},
\end{equation}
\begin{equation}\label{eq:LDj21}
 \mathfrak{L}_j\mathfrak{D}_{j-1}+\mathfrak{E}_j\mathfrak{U}_j=
{\mathfrak l}_{j}+{\mathfrak u}_{j}= -{\mathfrak p}_0^{(j)}.
\end{equation}

{\it Case 4.} In the case $\mathfrak{k}_j=1$,
$\mathfrak{k}_{j-1}=1$, $j\in\dN$ the equality~\eqref{eq:LDj21}
holds true and~\eqref{eq:LUj21} takes the form
\[
{\mathfrak l}_j{\mathfrak u}_{j-1}= {\mathfrak c}_{j-1}.
\]
In both cases the calculations in~\eqref{eq:uj} are still in force.
\end{proof}
\begin{thm}\label{Geronimus}
Let $\mathfrak{J}$ be a monic generalized Jacobi matrix associated
with $\mathfrak{F}\in{\bf D}_{-\infty}^+$ and let
$\mathfrak{J}=\mathfrak{L}\mathfrak{U}$ be its $LU$-factorization of
the form~\eqref{ind_LU}-\eqref{eq:LU_ind}. Then the matrix
${J}=\mathfrak{U}\mathfrak{L}$ is the monic Jacobi matrix associated
with $\lambda {\mathfrak F}(\lambda)+{\mathfrak s}_0\in{\bf N}_{-\infty}$.
\end{thm}
\begin{proof}
The fact that ${J}=\mathfrak{U}\mathfrak{L}$ is a monic Jacobi matrix follows
by straightforward calculations. Next, notice that
\begin{equation}\label{inCT_1}
{\mathfrak L}_{[0,j]}^{\top}e_0=e_0,\quad
\mathfrak{U}_{[0,j]}{G}_{[0,j]}e_0= \alpha e_0,\quad j\in\dZ_+,
\end{equation}
where
\[
\alpha=\left\{\begin{array}{cc}
             \epsilon_0\mathfrak{u}_0, & \mbox{if }{\mathfrak k}_0=1; \\
             \epsilon_0, & \mbox{if }{\mathfrak k}_0=2
           \end{array}\right.
\]
and the shortened matrices ${\mathfrak L}_{[0,j]}$,  ${\mathfrak
U}_{[0,j]}$, and $G_{[0,j]}$ are defined analogously
to~\eqref{ShortMat}. We can assume, without loss of generality, that
the function ${\mathfrak F}$ is normalized. Then it follows
from~\eqref{eq:4.5} that $\alpha={\mathfrak s}_1$ both for
${\mathfrak k}_0=1$ and ${\mathfrak k}_0=2$.

It follows from~\eqref{momrec} and~\eqref{ind_LU} that for $j$ big
enough  one has
\[
\begin{split}
\mathfrak{s}_k=
\left(e_0,\mathfrak{J}_{[0,j-1]}^kG_{[0,j-1]}e_0\right)=
\Big(e_0,\underbrace{\mathfrak{L}_{[0,j]}\mathfrak{U}_{[0,j]}
\dots \mathfrak{L}_{[0,j]}\mathfrak{U}_{[0,j]}}_{k\quad\text{times}} G_{[0,j]}e_0\Big)\\
=\Big(\mathfrak{L}^{\top}e_0,\underbrace{\mathfrak{U}_{[0,j]}\mathfrak{L}_{[0,j]}
\dots \mathfrak{U}_{[0,j]}\mathfrak{L}_{[0,j]}}_{k-1\quad\text{times}} \mathfrak{U}_{[0,j]}G_{[0,j]}e_0\Big),
\quad k\in\dN.
\end{split}
\]
Further, using~\eqref{inCT_1} we get
\[
\mathfrak{s}_k=\alpha\left(e_0,(\mathfrak{U}_{[0,j]}\mathfrak{L}_{[0,j]})^{k-1}e_0\right),
\]
which for sufficiently large $j$ can be rewritten as follows
\[
\left(e_0,{J}_{[0,{\mathfrak n}_{j+1}-1]}^{k-1}e_0\right)=
\frac{{\mathfrak s}_{k}}{{\mathfrak s}_1},\quad k\in\dN.
\]
This implies that the Jacobi matrix $J$ is associated with the
normalized function $F(\lambda)=\frac{\lambda {\mathfrak
F}(\lambda)+{\mathfrak s}_0}{{\mathfrak s}_1}\in{\bf N}_{-\infty}$
and, thus, also with ${\lambda {\mathfrak F}(\lambda)+{\mathfrak
s}_0}$.
\end{proof}

\subsection{$UL$-factorizations of Jacobi matrices}
\begin{prop}\label{prop:4.3}
Let ${J}$ be a monic Jacobi matrix associated with ${F}\in{\bf N}_{-\infty}$
which has the asymptotic
expansion~\eqref{asymp}, ${ s}_{-1}\in\dR$,
and let ${P}_j(\lambda)$ and ${ Q}_j(\lambda)$ be polynomials of the
first and the second kind, respectively, associated with the
sequence ${\bf s}=\{{ s}_{j}\}_{j=0}^\infty$. Then:
\begin{enumerate}
    \item[(i)] the normal indices ${\mathfrak
n}_j\,\,(j\in\dN)$ of the sequence ${\mathfrak s}=\{{
s}_{j-1}\}_{j=0}^\infty$ can be characterized by the conditions
\[
    \wh{{ P}}_{{\mathfrak n}_j-1}(0)\ne 0,\quad j\in\dN,
\]
where
$\wh{{ P}}_j(\lambda):= {Q}_j(\lambda)-{s}_{-1}{P}_j(\lambda)$;
    \item[(ii)] the Jacobi matrix ${J}$ admits the following factorization
\begin{equation}\label{UL}
{J}=\mathfrak{U}\mathfrak{L},
\end{equation}
where $\mathfrak{L}$ and $\mathfrak{U}$ are block lower and upper
triangular matrices having the
form~\eqref{eq:ind_LU}--\eqref{eq:ind_U}. Moreover, the following
relations hold true
\begin{equation}\label{f_UL_ind_1}
{\mathfrak l}_{j}=-\frac{\wh{{ P}}_{{\mathfrak n}_{j+1}-1}(0)}{\wh{{
P}}_{{\mathfrak
n}_j-1}(0)},\quad j\in\dZ_+.
\end{equation}
\end{enumerate}
\end{prop}
\begin{proof}
(i) The normal indices ${\mathfrak n}_j\,\,(j\in\dN)$ of the
sequence ${\mathfrak s}=\{{ s}_{j-1}\}_{j=0}^\infty$ can be
characterized by the conditions
\[
d_j:=\det
\begin{pmatrix}
{s}_{-1}&{s}_0&\dots &{s}_{{\mathfrak
n}_j-2}\\
{s}_{0}&{s}_1&\dots &{s}_{{\mathfrak
n}_j-1}\\
\hdotsfor{4}\\
{s}_{{\mathfrak n}_j-2}&{s}_{{\mathfrak n}_j-1}&\dots
&{s}_{{2\mathfrak n}_j-3}
\end{pmatrix}\ne 0,\quad j\in\dN.
\]

Now the first statement follows by the equalities
\[
    \begin{split}
d_{j} &={s}_{-1} \det\begin{pmatrix}
{s}_{1}&\dots &{s}_{\mathfrak{n}_j-1}\\
\hdotsfor{3}\\
{s}_{\mathfrak{n}_j-1}&\dots &{s}_{2\mathfrak{n}_j-3}
\end{pmatrix}+
\det\begin{pmatrix}
0&{s}_0&\dots &{s}_{\mathfrak{n}_j-2}\\
{s}_{0}&{s}_1&\dots &{s}_{\mathfrak{n}_j-1}\\
\hdotsfor{4}\\
{s}_{\mathfrak{n}_j-1}&{s}_{\mathfrak{n}_j+1}&\dots
&{s}_{2\mathfrak{n}_j-3}
\end{pmatrix}\\
&=(-1)^{{\mathfrak n}_j} \left({s}_{-1}{P}_{{\mathfrak
n}_j-1}(0)-{Q}_{{\mathfrak n}_j-1}(0)\right)
\det({s}_{i+k})_{i,k=0}^{\mathfrak{n}_j-1}\ne 0,
\end{split},
\]

 (ii) Comparing the matrix  
\[
\mathfrak{U}\mathfrak{L}=
\begin{pmatrix}
\mathfrak{U}_0\mathfrak{E}_0 +\mathfrak{D}_0\mathfrak{L}_1&  \mathfrak{D}_0\mathfrak{E}_1  & &\\
& & & \\
\mathfrak{U}_{1}\mathfrak{L}_1   &
\mathfrak{U}_1\mathfrak{E}_1+\mathfrak{D}_1\mathfrak{L}_2
& \mathfrak{D}_1\mathfrak{E}_2 &\\
 & & \\
       &\ddots &\ddots&\ddots\\
\end{pmatrix}
\]
with the matrix ${\mathfrak J}$ in~\eqref{monJac} one finds the
entries of ${\mathfrak L}$ and ${\mathfrak U}$. Let us  consider
four cases.

{\it Case 1.} $\mathfrak{k}_0=\mathfrak{k}_{1}=2$. In this case
$s_{-1}=\wh{{ P}}_{0}(0)=0$, $\wh{{ P}}_{1}(\lambda)\equiv 1$ and
$\wh{{ P}}_{2}(\lambda)=Q_2(\lambda)=\lambda-b_1$.

Since $b_1=-\wh{{ P}}_{2}(0)=0$ the equation
\begin{equation}\label{eq:UE0}
\mathfrak{U}_0\mathfrak{E}_0
+\mathfrak{D}_0\mathfrak{L}_1=\begin{pmatrix}
{\mathfrak e}_0& 1\\
{\mathfrak u}_0&0\\
\end{pmatrix}=
\begin{pmatrix}
b_0& 1\\
{\mathfrak c}_{1}&b_1\\
\end{pmatrix},
\end{equation}
is solvable and it follows from
\begin{equation}\label{eq:UL1}
\mathfrak{U}_{1}\mathfrak{L}_1=
\begin{pmatrix}
0&{\mathfrak l}_{1}\\
0 &0\\
\end{pmatrix}=
\begin{pmatrix} 0 & c_2\\
              0 & 0\end{pmatrix}
\end{equation}
that
\[
{\mathfrak l}_{1}=c_2=-\frac{\wh{{ P}}_{3}(0)}{\wh{{ P}}_{1}(0)}
\]

{\it Case 2.} $\mathfrak{k}_0=2$, $\mathfrak{k}_{1}=1$. In this case
$s_{-1}=\wh{{ P}}_{0}(0)=0$. It follows from the equation
\[
\mathfrak{U}_0\mathfrak{E}_0
+\mathfrak{D}_0\mathfrak{L}_1=\begin{pmatrix}
{\mathfrak e}_0& 1\\
{\mathfrak u}_0& {\mathfrak l}_{1}\\
\end{pmatrix}=
\begin{pmatrix}
b_0& 1\\
{\mathfrak c}_{1}&b_1\\
\end{pmatrix},
\]
that
\[
{\mathfrak l}_{1}=b_1=-\frac{\wh{{ P}}_{2}(0)}{\wh{{ P}}_{1}(0)}.
\]

{\it Case 3.} In the case $\mathfrak{k}_0=1$, $\mathfrak{k}_{1}=2$
one gets $\wh{{ P}}_{0}(0)\ne 0$, $\wh{{ P}}_{1}(0)=0$, and the
equations take the form
\[
\mathfrak{U}_0\mathfrak{E}_0 +\mathfrak{D}_0\mathfrak{L}_1=
{\mathfrak u}_0 = b_0,
\]
\[ 
\mathfrak{U}_{1}\mathfrak{L}_1=
\begin{pmatrix}
{\mathfrak l}_{1}\\
0\\
\end{pmatrix}=
\begin{pmatrix} c_1\\
              0 \end{pmatrix}.
\] 
Hence
\[
{\mathfrak l}_{1}=c_1=-\frac{\wh{{ P}}_{2}(0)}{\wh{{ P}}_{0}(0)}.
\]

{\it Case 4.} In the case $\mathfrak{k}_0=\mathfrak{k}_{1}=1$ one
has $s_{-1}=\wh{{ P}}_{0}(0)\ne 0$, and the equation
\[
\mathfrak{U}_0\mathfrak{E}_0 +\mathfrak{D}_0\mathfrak{L}_1=
{\mathfrak u}_0 +{\mathfrak l}_{1}= b_0,
\]
contains a free parameter ${\mathfrak u}_0$. Setting ${\mathfrak
u}_0=-\frac{1}{s_{-1}}$ one obtains
\[
{\mathfrak l}_{1}=b_0-{\mathfrak u}_0=-\frac{\wh{{ P}}_{1}(0)}{\wh{{
P}}_{0}(0)}.
\]

Assume now that~\eqref{f_UL_ind_1} is satisfied for some $j\in\dN$ and
 consider again four cases.

\noindent 1) $\mathfrak{k}_j=\mathfrak{k}_{j-1}=2$. 
These conditions can be rewritten as follows
\[
\wh{{ P}}_{\mathfrak{n}_{j}-1}(0)\ne 0,\quad \wh{{
P}}_{\mathfrak{n}_{j}+1}(0)\ne 0,\quad \wh{{
P}}_{\mathfrak{n}_{j}}(0)=0, \quad \wh{{
P}}_{\mathfrak{n}_{j}+2}(0)=0.
\]

This implies $b_{\mathfrak{n}_{j}+1}=0$ and hence the equations
\begin{equation}\label{eq:UEj}
\mathfrak{U}_j\mathfrak{E}_j
+\mathfrak{D}_j\mathfrak{L}_{j+1}=\begin{pmatrix}
{\mathfrak e}_j& 1\\
{\mathfrak u}_j&0\\
\end{pmatrix}=
\begin{pmatrix}
b_{\mathfrak{n}_{j}}& 1\\
{c}_{\mathfrak{n}_{j}}& b_{\mathfrak{n}_{j}+1}\\
\end{pmatrix},
\end{equation}
\[
\mathfrak{U}_{1}\mathfrak{L}_1=
\begin{pmatrix}
0&{\mathfrak l}_{{j}+1}\\
0 &0\\
\end{pmatrix}=
\begin{pmatrix} 0 & c_{\mathfrak{n}_{j}+1}\\
              0 & 0\end{pmatrix}
\]
are solvable and
\begin{equation}\label{fTh55_1}
{\mathfrak l}_{j+1}=c_{\mathfrak{n}_{j}+1}=-\frac{\wh{{
P}}_{\mathfrak{n}_{j+2}-1}(0)}{\wh{{ P}}_{\mathfrak{n}_{j+1}-1}(0)}
\end{equation}

\noindent 2) $\mathfrak{k}_{j-1}=2$, $\mathfrak{k}_{j}=1$. In this
case
\[
\wh{{ P}}_{\mathfrak{n}_{j}-1}(0)\ne 0,\quad \wh{{
P}}_{\mathfrak{n}_{j}}(0)= 0,\quad \wh{{
P}}_{\mathfrak{n}_{j}+1}(0)\ne 0, \quad \wh{{
P}}_{\mathfrak{n}_{j}+2}(0)\ne 0.
\] It follows from the equation
\[
\mathfrak{U}_j\mathfrak{E}_j
+\mathfrak{D}_j\mathfrak{L}_{j+1}=\begin{pmatrix}
{\mathfrak e}_j& 1\\
{\mathfrak u}_j& {\mathfrak l}_{j+1}\\
\end{pmatrix}=
\begin{pmatrix}
b_{\mathfrak{n}_{j}}& 1\\
{c}_{\mathfrak{n}_{j}}& b_{\mathfrak{n}_{j}+1}\\
\end{pmatrix},
\]
that
\begin{equation}\label{fTh55_2}
{\mathfrak l}_{j+1}=b_{\mathfrak{n}_{j}+1}=-\frac{\wh{{
P}}_{\mathfrak{n}_{j+2}-1}(0)}{\wh{{ P}}_{\mathfrak{n}_{j+1}-1}(0)}.
\end{equation}

\noindent 3) In the case $\mathfrak{k}_{j-1}=1$,
$\mathfrak{k}_{j}=2$ one obtains
\[ \wh{{ P}}_{\mathfrak{n}_{j}-1}(0)\ne 0,\quad
\wh{{ P}}_{\mathfrak{n}_{j}}(0)\ne 0,\quad \wh{{
P}}_{\mathfrak{n}_{j}+1}(0)= 0, \quad \wh{{
P}}_{\mathfrak{n}_{j}+2}(0)\ne 0.
\] It follows that
\[
\mathfrak{U}_j\mathfrak{E}_j +\mathfrak{D}_j\mathfrak{L}_{j+1}=
{\mathfrak u}_j= b_{\mathfrak{n}_{j}}.
\]
Hence
\begin{equation}\label{fTh55_3}
{\mathfrak l}_{j+1}=c_{\mathfrak{n}_{j}+1}=-\frac{\wh{{
P}}_{\mathfrak{n}_{j+2}-1}(0)}{\wh{{ P}}_{\mathfrak{n}_{j+1}-1}(0)}.
\end{equation}

\noindent 4) In the case $\mathfrak{k}_{j-1}=\mathfrak{k}_{j}=1$ one has
\[
\wh{{ P}}_{\mathfrak{n}_{j}-1}(0)\ne 0,\quad \wh{{
P}}_{\mathfrak{n}_{j}}(0)\ne 0,\quad \wh{{
P}}_{\mathfrak{n}_{j}+1}(0)\ne 0, \quad 
\]
and the equations
\[
\mathfrak{U}_j\mathfrak{L}_j = {\mathfrak u}_j {\mathfrak l}_{j}=
c_{\mathfrak{n}_{j}-1},
\]
\begin{equation}\label{eq:LUj2}
\mathfrak{U}_j\mathfrak{E}_j +\mathfrak{D}_j\mathfrak{L}_{j+1}=
{\mathfrak u}_j +{\mathfrak l}_{j+1}= b_{\mathfrak{n}_{j}},
\end{equation}
yield
\begin{equation}\label{fTh55_4}
\begin{split}
{\mathfrak l}_{j+1}&= b_{\mathfrak{n}_{j}}-{\mathfrak u}_j=
b_{\mathfrak{n}_{j}}-\frac{ c_{\mathfrak{n}_{j}-1}}{{\mathfrak
l}_{j}}\\
&=\frac{b_{\mathfrak{n}_{j}}\wh P_{\mathfrak{n}_{j}}(0)+
c_{\mathfrak{n}_{j}-1}\wh P_{\mathfrak{n}_{j}-1}(0)}{\wh
P_{\mathfrak{n}_{j}}(0)}= -\frac{\wh
P_{\mathfrak{n}_{j+2}-1}(0)}{\wh P_{\mathfrak{n}_{j+1}-1}(0)}.
\end{split}
\end{equation}
This completes the proof of~\eqref{f_UL_ind_1}.
\end{proof}

\begin{rem}
It should be noted that, actually, both the matrices in the
$UL$-decom\-position~\eqref{UL} depend on ${s}_{-1}$, that is,
$\mathfrak{U}=\mathfrak{U}({s}_{-1})$ and
$\mathfrak{L}=\mathfrak{L}({s}_{-1})$.
\end{rem}

\begin{thm}\label{Geron}
Let $J$ be a monic Jacobi matrix associated with $F\in{\mathbf
N}_{-\infty}$, ${ s}_{-1}\in\dR$, and let
$J=\mathfrak{U}\mathfrak{L}$ be the corresponding $UL$
factorization of $J$ of the form~\eqref{eq:ind_LU}--\eqref{eq:ind_U}. Then 
$\mathfrak{J}=\mathfrak{L}\mathfrak{U}$ is the monic generalized
Jacobi matrix associated with
$\frac{F(\lambda)}{\lambda}-\frac{s_{-1}}{\lambda}\in{\bf
D}_{0,-\infty}^-$.
\end{thm}
\begin{proof}
The fact that $\mathfrak{J}$ is the monic generalized
Jacobi matrix can be easily verified by straightforward calculations.
The rest of the proof can be done by reversing the reasoning given
in the proof of Theorem~\ref{Geronimus}.
\end{proof}

The transform $J=\mathfrak{U}\mathfrak{L}\mapsto
\mathfrak{J}=\mathfrak{L}\mathfrak{U}$ in Theorem~\ref{Geron} is
called the {\it Geronimus transform} of the Jacobi matrix $J$ with
the parameter $s_{-1}$.
\subsection{Jacobi matrices associated with generalized Stieltjes functions}
Let $\kappa$ be a nonnegative integer. Remind that a function
$F$, meromorphic in $\dC\setminus\dR$, is said to belong to the
class ${\mathbf N}_\kappa$ if the domain of holomorphy $\rho(F)$ of
the function $F$ is symmetric with respect to $\dR$,
$F(\bar\lambda)=\overline{F(\lambda)}$ for $\lambda\in\rho(F)$, and
the kernel
\[
\begin{cases}
       {\sf N}_{F}(\lambda,\omega)=
\frac{F(\lambda)-\overline{F(\omega)}}{\lambda-\overline{\omega}},& \text{$\lambda,\omega\in\rho(F)$};\\
  {\sf N}_{F}(\lambda,\overline{\lambda})=F'(\lambda),& \text{$\lambda\in\rho(F)$}\\
  \end{cases}
\]
has  $\kappa$ negative squares on  $\rho(F)$. The last statement
means that for every $n\in\dN$ and
$\lambda_{1},\lambda_{2},\dots,\lambda_{n}\in\rho(F)$, $n\times n$
matrix $({\sf N}_{F}(\lambda_{i},\lambda_{j}))_{i,j=1}^n$ has at
most $\kappa $ negative eigenvalues (with account of multiplicities)
and for some choice of $n$,
$\lambda_{1},\lambda_{2},\dots,\lambda_{n}$ it has exactly $\kappa$
negative eigenvalues (see~\cite{KL79}). Clearly, ${\mathbf
N}_0={\mathbf N}$.
\begin{defn}\label{deffunctSpm}
Let us say that a function $F$ holomorphic in $\dC_+$ belongs to the
generalized Stieltjes class ${\bf S}^{\pm\kappa}$  if
$F(\lambda)\in{\bf N}$ and $\lambda^{\pm 1} F(\lambda)\in{\bf
N}_{\kk}$. Let us set
\[
{\bf S}^{\pm\kappa}_{-\infty}={\bf S}^{\pm\kappa}\cap {\bf
N}_{-\infty}.
\]
\end{defn}
According to~\cite[Theorem 2.4]{DM97} any function $F\in{\bf
S}^{\kappa}_{-\infty}$ admits the integral representation
\begin{equation}\label{eq:4.24}
    F(\lambda)=\sum_{j=1}^{p}\frac{A_j}{t_j-\lambda}+\int_0^\infty\frac{d\sigma(t)}{t-\lambda},
\end{equation}
where $A_j>0$ ($j=1,\dots,p$), $p=\kappa$ and $\sigma$ is a finite
measure on $[0,+\infty)$, such that
\begin{equation}\label{eq:4.25}
    \int_0^\infty t^{2n}d\sigma(t)<\infty \mbox{ for all }n\in\dN.
\end{equation}

Similarly, any function $F\in{\bf S}^{-\kappa}_{-\infty}$ admits the
integral representation~\eqref{eq:4.24}, where  $A_j>0$
($j=1,\dots,p$), $\sigma$ is a finite measure on $[0,+\infty)$,
which satisfies~\eqref{eq:4.25} and
\begin{equation}\label{eq:4.24A}
p=\left\{\begin{array}{cl}
           \kappa-1, & \mbox{ if }0< F(0-)\le \infty; \\
           \kappa, & \mbox{ if }F(0-)\le 0.\\
         \end{array}
         \right.
\end{equation}
This implies, in particular, that every function $F\in{\bf
S}^{-\kappa}_{-\infty}$ belongs either to ${\bf
S}^{\kappa}_{-\infty}$ or to ${\bf S}^{\kappa+1}_{-\infty}$.

\begin{cor}
Let $F\in{\bf S}^{\kappa}_{-\infty}$ have the asymptotic
expansion~\eqref{asymp}, let $J$ be a monic Jacobi matrix associated
with $F$, let $\{{\mathfrak n}_j\}_{j=1}^\infty$ be the set of
normal indices of the sequence ${\mathfrak s}
=\{{ s}_{j+1}\}_{j=0}^\infty$, and let $J=LU$
be its $LU$-factorization of the form~\eqref{LU}-\eqref{eq:LU}.
Then:
\begin{enumerate}
    \item[(i)] the sequence
${\mathfrak k}_j:={\mathfrak n}_{j+1}-{\mathfrak n}_{j}$,
$j\in\dZ_+$ is stabilized, i.e. there is $N>0$ such that
${\mathfrak k}_j=1$ for $j\ge N$;
    \item[(ii)] the matrix $\mathfrak{J}=UL$ is the monic generalized
Jacobi matrix of the form
\[
\mathfrak{J}=\left(%
\begin{array}{cc}
  \mathfrak{J}_{[0,N]} &  \mathfrak{J}_{12}  \\
  \mathfrak{J}_{21} &  \mathfrak{J}_{[N+1,\infty)} \\
\end{array}%
\right),\quad 
\mathfrak{J}_{12}=\left(%
\begin{array}{cc}
   & {\bf 0} \\
  1 &  \\
\end{array}%
\right),\,\,
\mathfrak{J}_{21}=\left(%
\begin{array}{cc}
   & c_N \\
  {\bf 0} &  \\
\end{array}%
\right),
\]
$\mathfrak{J}_{[N+1,\infty)}$ is a monic Jacobi matrix,
$\mathfrak{J}_{[0,N]}$ is a monic generalized Jacobi matrix of the
form~\eqref{mJacobi};
   \item[(iii)] the matrix $\mathfrak{J}=UL$ is
associated with $\lambda F(\lambda)+s_0\in{\bf D}_{-\infty}^-$.
 \end{enumerate}
\end{cor}

\section{Generalized Cholesky decomposition}

Recall that the Cholesky decomposition is a decomposition of a symmetric positive-definite real matrix $A$ as the product of a lower triangular matrix
$L$ and its transpose $L^{\top}$, that is $A=LL^{\top}$. The decomposition can be also rewritten in the form
$A=L\Lambda L^{\top}$, where $\Lambda$ is a diagonal matrix and all the diagonal entries of $L$ are
equal to $1$. In this section, we give a generalization of the latter decomposition
to  an arbitrary symmetric Jacobi matrix $J_s$.

At first, note that every monic Jacobi matrix $J$ can be reduced to a symmetric Jacobi matrix via the following transformation
\begin{equation}\label{symtomon}
J_s=\Psi^{-1} J \Psi,
\end{equation}
where $\Psi=\diag(1,\sqrt{c_0}, \sqrt{c_0c_1},\sqrt{c_0c_1c_2}, \dots )$; here the square root is chosen to be positive.
In fact, every symmetric Jacobi matrix can be represented in such a way.
So, now we can reformulate Proposition~\ref{prop31} in an appropriate way.
\begin{prop}
Let $J_s$ be a symmetric Jacobi matrix associated with $F\in{\mathbf
N}_{-\infty}$ and let ${\mathfrak k}_j:={\mathfrak
n}_{j+1}-{\mathfrak n}_{j}$,  $j\in\dZ_+$, where $\mathfrak{n}_0=0$
and $\{{\mathfrak n}_j\}_{j=1}^\infty$ is the set of normal indices
of the sequence ${\mathfrak s}=\{{\mathfrak s}_{j}\}_{j=0}^\infty$
defined by the asymptotic expansion~\eqref{asympF} of the function
$\lambda F(\lambda)+s_0$. Then $J$
admits the following generalized Cholesky decomposition
\begin{equation}\label{Cholesky}
J_s=L\Lambda L^{\top},
\end{equation}
where $L$ is a   block lower triangular matrix having the form
\[
    L=\begin{pmatrix}
I_{\mathfrak{k}_{0}} & 0&       &\\
\widehat{L}_{1}   & I_{\mathfrak{k}_{1}}   & 0 &\\
        &\widehat{L}_2    & I_{\mathfrak{k}_{2}} &\ddots\\
&       &\ddots &\ddots\\
\end{pmatrix},
\]
 in which the sub-diagonal of $L$ consists of
$\mathfrak{k}_{j+1}\times\mathfrak{k}_{j}$ matrices
\[
\widehat{L}_{j+1}=
\begin{pmatrix}
\hat{l}_{j+1}&0\\
0&0\\
\end{pmatrix},\quad
\begin{pmatrix}
\hat{l}_{j+1}\\
0\\
\end{pmatrix},\quad
\begin{pmatrix}
\hat{l}_{j+1}&0\\
\end{pmatrix},\quad
(\hat{l}_{j+1}),
\]
and $\Lambda=\diag(\Lambda_0,\Lambda_1,\dots )$ is a block diagonal matrix with the entries
\[
{\Lambda}_j=\left\{\begin{array}{cc}
             {\Lambda}_0^{(j)}, & \mbox{if }{\mathfrak k}_j=1; \\
             \begin{pmatrix} 0 & {\Lambda}_0^{(j)}\\
              {\Lambda}_0^{(j)}& {\Lambda}_1^{(j)}\end{pmatrix} & \mbox{if }{\mathfrak k}_j=2,
           \end{array}\right.
\]
 Moreover, the factorization~\eqref{Cholesky} is unique.
\end{prop}

\begin{proof}
Making use of~\eqref{LU} and~\eqref{symtomon} gives
\begin{equation}\label{Chol1}
 J_s=\Psi^{-1} LU \Psi.
\end{equation}
Further, let $\Psi=\diag(\Psi_0,\Psi_1,\Psi_2,\dots)$ be  a partition corresponding to the one of $L$. Then~\eqref{Chol1}
can be rewritten as follows
\begin{equation}\label{Chol2}
\begin{pmatrix}
b_0 &\sqrt{c_0}&       &\\
\sqrt{c_0}&b_{1}    &{\sqrt{c_1}}&\\
        &\sqrt{c_1}&{b}_{2} &\ddots\\
&       &\ddots &\ddots\\
\end{pmatrix}=
\begin{pmatrix}
\widehat{\Lambda}_0 & \widehat{D}_0&       &\\
\widehat{L}_{1}\widehat{\Lambda}_0   & \widehat{L}_1\widehat{D}_0+\widehat{\Lambda}_1   & \widehat{D}_1 &\\
        &\widehat{L}_2\widehat{\Lambda}_1    & \widehat{L}_2\widehat{D}_0+\widehat{\Lambda}_2 &\ddots\\
&       &\ddots &\ddots\\
\end{pmatrix},
\end{equation}
where
\[
\widehat{\Lambda}_j=\Psi_{j}^{-1}U_{j}\Psi_{j},\quad
\widehat{D}_j=\Psi_{j}^{-1}D_{j}\Psi_{j+1},\quad
\widehat{L}_{j+1}=\Psi_{j+1}^{-1}L_{j+1}\Psi_{j},\quad j\in\dZ_+.
\]
Since the matrices in~\eqref{Chol2} are symmetric we have that $\widehat{\Lambda}_{j}^{\top}=\widehat{\Lambda}_{j}$ and
$\widehat{D}_j=(\widehat{L}_{j+1}\widehat{\Lambda}_{j})^{\top} =\widehat{\Lambda}_j\widehat{L}_{j+1}^{\top}$ for $j\in\dZ_+$.
This observation and the representation~\eqref{Chol2} yield~\eqref{Cholesky}.
\end{proof}

\section{Convergence of Pad\'e approximants for ${\bf D}_{-\infty}^\pm$ functions}

In this section we restrict our consideration to the case of probability measures supported
on $[-1,1]$. As a consequence, the corresponding monic classical Jacobi matrices will be bounded.
To specify the results of this section, let us recall the notion of Pad\'e approximants.
Suppose we are given a formal power series
$F(\lambda)=\displaystyle{-\sum\limits_{j=0}^{\infty}\frac{s_{j}}{{\lambda}^{j+1}}}$
with $s_j\in\dR$. Let $L$, $M$ be
positive integers. Then  an $[L/M]$ Pad\'e approximant for $F$ is defined
as a ratio
\[
F^{[L/M]}(\lambda)=\frac{A^{[L/M]}\left(\frac{1}{\lambda}\right)}
{B^{[L/M]}\left(\frac{1}{\lambda}\right)}
\]
of polynomials $A^{[L/M]}$, $B^{[L/M]}$ of formal degree $L$ and
$M$, respectively, such that $B^{[L/M]}(0)\ne 0$ and
\begin{equation}\label{eq:1.1}
    \sum_{j=0}^{L+M-1}\frac{s_{j}}{{\lambda}^{j+1}}+F^{[L/M]}(\lambda)= O\left(\frac{1}{{\lambda}^{L+M+1}}\right),\quad \lambda\to\infty.
\end{equation}
More information about Pad\'e approximants can be found in~\cite{Baker}.
In what follows $\widehat{\dC}$ is considered as $\dC\cup\{\infty\}$ equipped with
the spherical metric.

\subsection{Pad\'e approximants for ${\bf D}_{-\infty}^-$ functions}

In a previous work we proved the following result.
\begin{thm}[\cite{DD09}]\label{PA_Dminus}
Let $\sigma$ be a finite nonnegative measure on $E=[-1,\alpha]\cup[\beta,1]$,
$0\in[\alpha,\beta]$, and let
\begin{equation}\label{PartCase2}
    {\mathfrak F}(\lambda)=\int_E\frac{td\sigma(t)}{t-\lambda}.
\end{equation}
Then ${\mathfrak F}\in{\bf D}_{-\infty}^-$ and the sequence $\{{\mathfrak
F}^{[\mathfrak{n}_j/\mathfrak{n}_j]}\}_{j=0}^{\infty}$ converges to ${\mathfrak F}$ locally uniformly in
$\widehat{\dC}\setminus ([-1-\varepsilon,\alpha]\cup[\beta, 1+\varepsilon])$
for some $ \varepsilon\ge0$ if and only if the
condition
\begin{equation}\label{Pcond}
\sup\limits_{j\in\dZ_+}\left|\frac{P_{\mathfrak{n}_j+1}(0)}{P_{\mathfrak{n}_j}(0)}\right|<\infty
\end{equation}
is fulfilled for the polynomials $P_j$ orthogonal with respect to $\sigma$.
\end{thm}

\begin{rem}
 In~\cite[Section~5.3]{DD09} we gave some sufficient conditions for the measure $\sigma$
to possess the property~\eqref{Pcond}.
\end{rem}

This result can be interpreted in terms of the Darboux transformations as follows.
\begin{thm}\label{DarbouxPade}
Let ${\mathfrak F}\in{\bf D}_{-\infty}^-$ have the form~\eqref{PartCase2}.
Then the following statements are equivalent:
\begin{enumerate}
\item[(i)] for the number
\[
\mathfrak{s}_{-1}:=\int_Ed\sigma(t),
\]
the operator ${U}={U}(\mathfrak{s}_{-1})$
in the factorization~\eqref{ind_UL} of $\mathfrak{J}$ corresponding to
${\mathfrak F}$
is bounded in $\ell^2_{[0,\infty)}$;
\item[(ii)] the operator $U$ in the factorization~\eqref{LU} of $J$
corresponding to the function
\[
{F}(\lambda)=\int_E\frac{d\sigma(t)}{t-\lambda}
\]
is bounded in $\ell^2_{[0,\infty)}$;
\item[(iii)]the sequence $\{{\mathfrak
F}^{[\mathfrak{n}_j/\mathfrak{n}_j]}\}_{j=0}^{\infty}$ of diagonal Pad\'e
approximants converges to ${\mathfrak F}$ locally uniformly in
$\widehat{\dC}\setminus ([-1-\varepsilon,\alpha]\cup[\beta, 1+\varepsilon])$
for some $ \varepsilon\ge 0$.
\end{enumerate}
\end{thm}
\begin{proof}
Clearly, the first statement is equivalent to the second one by construction.
So, it remains to prove the equivalence of (ii) and (iii).
First, suppose that (iii) holds true and, therefore, from Theorem~\ref{PA_Dminus}
we see that~\eqref{Pcond} is fulfilled. Then, according
to~\eqref{fTh52_1}, \eqref{fTh52_2}, \eqref{fTh52_3}, and~\eqref{fTh52_4},
the operator $U$ is bounded. The implication from (ii) to (iii) follows from
formulas~\eqref{fTh52_3},~\eqref{fTh52_4} and Theorem~\ref{PA_Dminus}.
\end{proof}
\begin{rem}
Under the condition of Theorem~\ref{DarbouxPade}, from~\eqref{fTh52_1},~\eqref{fTh52_2},~\eqref{eq_c21},~\eqref{fTh52_3},
and~\eqref{eq_c11} we see that the operators $U$ and $L$, given by~\eqref{LU}, are bounded or unbounded simultaneously.
\end{rem}

In some cases, we can specialize Theorem~\ref{DarbouxPade}.

\begin{cor}
Let ${\mathfrak F}$ admit the following representation
\begin{equation}\label{n_kappa_+}
{\mathfrak F}(\lambda)= \int_{0}^{1}\frac{td\sigma(t)}{t-\lambda}+
\sum_{j=1}^{p}\frac{a_jt_j}{t_j-\lambda},
\end{equation}
where $\sigma$ is a finite nonnegative measure on $[0,1]$,
$t_j\le 0$, and $a_j\ge 0$ for $j=1,\dots, p$.
Then the following statements hold true:
\begin{enumerate}
\item[(i)]the sequence $\{{\mathfrak
F}^{[\mathfrak{n}_j/\mathfrak{n}_j]}\}_{j=0}^{\infty}$ of diagonal
Pad\'e approximants converges to ${\mathfrak F}$ locally uniformly
in $\widehat{\dC}\setminus ([0,1]\cup\{t_k\}_{k+1}^{p})$;
\item[(ii)]
the polynomials $P_j$ orthogonal with respect to $\sigma+\sum_{j=1}^{p}a_j\Theta(\cdot-t_j)$
possess the following property
\[
\sup\limits_{j\in\dZ_+}\left|\frac{P_{\mathfrak{n}_j+1}(0)}{P_{\mathfrak{n}_j}(0)}\right|<\infty
\]
(here $\Theta$ denotes the Heaviside step function);
\item[(iii)] for the number
\[
\mathfrak{s}_{-1}:=\int_0^1d\sigma(t)+\sum_{j=1}^{p}a_j,
\]
the operator ${U}={U}(\mathfrak{s}_{-1})$ in the
factorization~\eqref{ind_UL} of $\mathfrak{J}$ corresponding to
${\mathfrak F}$ is bounded in $\ell^2_{[0,\infty)}$;
\item[(iv)] the operator $U$ in the factorization~\eqref{LU} of $J$
corresponding to the function
\begin{equation}\label{s_kappa_+}
{F}(\lambda)=\int_0^1\frac{d\sigma(t)}{t-\lambda}+
\sum_{j=1}^{p}\frac{a_j}{t_j-\lambda}
\end{equation}
is bounded in $\ell^2_{[0,\infty)}$.
\end{enumerate}
\end{cor}
\begin{proof}
Clearly, ${\mathfrak F}\in{\bf D}_{-\infty}^-$.
In fact, statement (i) for the class in question was proved in~\cite{Rakh}.
Statement (ii) is an immediate consequence of statement (i) and Theorem~\ref{PA_Dminus}.
The rest follows from Theorem~\ref{DarbouxPade}.
\end{proof}
\begin{rem}
 As was mentioned in Subsection 4.3  the function
$F$ in~\eqref{s_kappa_+} belongs to the class ${\bf S}^{-p}$.
\end{rem}

\subsection{Pad\'e approximants for ${\bf D}_{-\infty}^+$ functions}

In this section we present convergence results for diagonal Pad\'e approximants to
${\bf D}_{-\infty}^+$ functions.

By re-examining~\cite[Theorem~4.16]{DD07}, \cite[Theorem~3.3]{DD09}
and~\cite[Theorem~5.5]{DD09}, we arrive at the following more general result.

\begin{prop}
\label{quasipol}
Let 
\[
F(\lambda)=\int_{-1}^{1}\frac{d\sigma(t)}{t-\lambda}.
\]
Then for any sequence $\{\tau_j\}_{j=1}^{\infty}$ of real numbers
the following relation holds true
\begin{equation}\label{modPA_as}
\widehat{F}^{[\mathfrak{n}_j/\mathfrak{n}_j]}:=
-\frac{Q_{\mathfrak{n}_j}(\lambda)+\tau_jQ_{\mathfrak{n}_j-1}(\lambda)}
{P_{\mathfrak{n}_j}(\lambda)+\tau_jP_{\mathfrak{n}_j-1}(\lambda)}
= -\sum_{i=0}^{2\mathfrak{n}_j-2}\frac{s_{i}}{\lambda^{i+1}}+
                O\left(\frac{1}{\lambda^{2\mathfrak{n}_j}}\right),\quad
                \lambda\to\infty,
\end{equation}
where  $P_j$ and $Q_j$ are  polynomials of the first and second
kind, respectively, corresponding to $\sigma$. Moreover, the
sequence of the modified Pad\'e approximants
$\widehat{F}^{[\mathfrak{n}_j/\mathfrak{n}_j]}$ converges to ${F}$
locally uniformly in $\widehat{\dC}\setminus
([-1-\varepsilon,1+\varepsilon])$ for some $ \varepsilon\ge0$ if and
only if the condition
\begin{equation}\label{Pcond2}
\sup\limits_{j\in\dN}\left|\tau_j\right|<\infty
\end{equation}
is satisfied.
\end{prop}
\begin{proof}
The relation~\eqref{modPA_as} is an appropriate reformulation of the results of~\cite[Section~I.4.2]{Ach61}.
Next, notice that the symmetric Jacobi matrix ${\cJ}$
associated with $\sigma$ is a bounded linear operator in ${\ell}^2_{[0,\infty)}$.
Furthermore, using the classical versions of formulas~\eqref{polynom2},~\eqref{SWfun},
and~\eqref{mQP}, the modified Pad\'e approximant can be rewritten as follows
\begin{equation}\label{modPA_repr}
\begin{split}
\widehat{F}^{[\mathfrak{n}_j/\mathfrak{n}_j]}(\lambda)&=
-\frac{Q_{\mathfrak{n}_j}(\lambda)+\tau_jQ_{\mathfrak{n}_j-1}(\lambda)}
{P_{\mathfrak{n}_j}(\lambda)+\tau_jP_{\mathfrak{n}_j-1}(\lambda)}
=-\frac{\det(\lambda-\cJ_{[1,\mathfrak{n}_j-1]}^{(\tau_j)})}
{\det(\lambda-\cJ_{[0,\mathfrak{n}_j-1]}^{(\tau_j)})}\\
&=\left((\cJ_{[0,\mathfrak{n}_j-1]}^{(\tau_j)}-\lambda)^{-1}e,e\right),
\end{split}
\end{equation}
where $\cJ_{[0,\mathfrak{n}_j-1]}^{(\tau_j)}$ is a rank-one perturbation
of $\cJ_{[0,\mathfrak{n}_j-1]}$ of the form
\[
\cJ_{[0,\mathfrak{n}_j-1]}^{(\tau_j)}=
\begin{pmatrix}
{b}_{0}   & \sqrt{c}_{0} &       &\\
\sqrt{c}_{0}   &\ddots    &\ddots&\\
        &\ddots    & \ddots& \sqrt{c}_{\mathfrak{n}_j-2}\\
&       &\sqrt{c}_{\mathfrak{n}_j-2} &b_{\mathfrak{n}_j-1}+\tau_j\\
\end{pmatrix}=\cJ_{[0,\mathfrak{n}_j-1]}+\diag\{0,\dots,0,\tau_j\}
\]
(cf.~\cite[Proposition~5.8]{Si}). Now, let us suppose that~\eqref{Pcond2} is satisfied. Then one
obtains
\begin{equation}\label{eq:est}
\Vert \cJ_{[0,\mathfrak{n}_j-1]}^{(\tau_j)}\Vert\le\Vert
\cJ_{[0,\mathfrak{n}_j-1]}\Vert+|\tau_{j}|\le 1+\varepsilon,\quad j\in\dZ_+,
\end{equation}
for some $\varepsilon>0$.
It follows from~\eqref{modPA_repr}, the inequalities
\[
\Vert(\cJ_{[0,\mathfrak{n}_j-1]}^{(\tau_j)}-\lambda)^{-1}\Vert \le\frac{1}
{|\lambda|-\|\cJ_{[0,\mathfrak{n}_j-1]}^{(\tau_j)}\|}\quad
(|\lambda|>\|\cJ_{[0,\mathfrak{n}_j-1]}^{(\tau_j)}\|)
\]
and~\eqref{eq:est} that
\[
\begin{split}
\left|\widehat{F}^{[\mathfrak{n}_j/\mathfrak{n}_j]}(\lambda)\right|&=
\left|\left((\cJ_{[0,\mathfrak{n}_j-1]}^{(\tau_j)}-\lambda)^{-1}e,e\right)_{\ell^2}\right|\\
&\le\Vert(\cJ_{[0,\mathfrak{n}_j-1]}^{(\tau_j)}-\lambda)^{-1}e\Vert_{\ell^2}\Vert
e\Vert_{\ell^2} \le \frac{1}{|\lambda|-1-\varepsilon}.
\end{split}
\]
for $|\lambda|>1+\varepsilon$.
The pointwise convergence $\widehat{F}^{[\mathfrak{n}_j/\mathfrak{n}_j]}(\lambda)$ to
$F(\lambda)$ for $\lambda\in\dC\setminus\dR$ follows from the proof of~\cite[Proposition~5.3]{Si}.
Finally, it remains to apply the Vitali theorem.

Let us prove the necessity by proving the
contrary statement. So, suppose that
\[
\sup\limits_{j\in\dN}\left|\tau_j\right|=\infty.
\]
Further, note that the poles of the modified Pad\'e approximant
$\widehat{F}^{[\mathfrak{n}_j/\mathfrak{n}_j]}$ coincide with eigenvalues of the
matrix $\cJ_{[0,\mathfrak{n}_j-1]}^{(\tau_j)}$. Taking into account that
$\cJ_{[0,\mathfrak{n}_j-1]}^{(\tau_j)}$ is a self-adjoint matrix, one obtains
\[
|\lambda_{max}(\cJ_{[0,\mathfrak{n}_j-1]}^{(\tau_j)})|=
\Vert \cJ_{[0,\mathfrak{n}_j-1]}^{(\tau_j)}\Vert
\ge|(\cJ_{[0,\mathfrak{n}_j-1]}^{(\tau_j)}e_{\mathfrak{n}_j-1},e_{\mathfrak{n}_j-1})|
=\left|b_{\mathfrak{n}_j-1}+\tau_j\right|,
\]
where $\lambda_{max}(\cJ_{[0,\mathfrak{n}_j-1]}^{(\tau_j)})$ is the eigenvalue of the
matrix $\cJ_{[0,\mathfrak{n}_j-1]}^{(\tau_j)}$ with the largest absolute value.
Since the sequence $\{b_k\}_{k=0}^{\infty}$ is bounded, we
have that infinity is an accumulation point of the set of all
poles of the modified Pad\'e approximants $\widehat{F}^{[\mathfrak{n}_j/\mathfrak{n}_j]}$.
\end{proof}

\begin{thm}\label{PA_Dplus}
Let $\sigma$ be a finite nonnegative measure on $[-1,1]$ and let
\begin{equation}\label{fDplus}
    {\mathfrak F}(\lambda)=\frac{1}{\lambda}\int_{-1}^{1}\frac{d\sigma(t)}{t-\lambda}-\frac{s_{-1}}{\lambda},
\end{equation}
where $s_{-1}$ is a fixed real number.
Then ${\mathfrak F}\in{\bf D}_{-\infty}^+$ and the sequence $\{{\mathfrak
F}^{[\mathfrak{n}_j/\mathfrak{n}_j]}\}_{j=0}^{\infty}$ converges to ${\mathfrak F}$ locally uniformly in
$\widehat{\dC}\setminus ([-1-\varepsilon, 1+\varepsilon])$
for some $ \varepsilon\ge0$ if and only if the
condition
\begin{equation}\label{Pcond3}
\sup\limits_{j\in\dZ_+}\left|\frac{Q_{\mathfrak{n}_j}(0)-s_{-1}P_{\mathfrak{n}_j}(0)}
{Q_{\mathfrak{n}_j-1}(0)-s_{-1}P_{\mathfrak{n}_j-1}(0)}\right|<\infty
\end{equation}
is fulfilled for the polynomials $P_j$ and $Q_j$ of the first and second kind
corresponding to $\sigma$.
\end{thm}
\begin{proof}
Observe that by setting
\[
\tau_{j}=-\frac{Q_{\mathfrak{n}_j}(0)-s_{-1}P_{\mathfrak{n}_j}(0)}
{Q_{\mathfrak{n}_j-1}(0)-s_{-1}P_{\mathfrak{n}_j-1}(0)}
\]
we have that ${\mathfrak F}^{[\mathfrak{n}_j/\mathfrak{n}_j]}(0)=s_{-1}$. Thus,
due to~\eqref{eq:1.1} and~\eqref{modPA_as} one has the following representation
of the diagonal Pad\'e approximants to
${\mathfrak F}$
\[
{\mathfrak F}^{[\mathfrak{n}_j/\mathfrak{n}_j]}(\lambda)=\frac{1}{\lambda}\left(
\frac{Q_{\mathfrak{n}_j}(\lambda)+\tau_jQ_{\mathfrak{n}_j-1}(\lambda)}
{P_{\mathfrak{n}_j}(\lambda)+\tau_jP_{\mathfrak{n}_j-1}(\lambda)}-s_{-1}\right).
\]
Now the statement is an immediate consequence of Proposition~\ref{quasipol}.
\end{proof}

The above result can be reformulated in terms of the Darboux transformations as follows.
\begin{thm}\label{ThOnPD}
Let ${\mathfrak F}\in{\bf D}_{-\infty}^+$ have the form~\eqref{PartCase2}.
Then the following statements are equivalent:
\begin{enumerate}
\item[(i)]
the operator $\mathfrak{L}$
in the factorization~\eqref{ind_LU} of $\mathfrak{J}$ corresponding to
${\mathfrak F}$
is bounded in $\ell^2_{[0,\infty)}$;
\item[(ii)] the operator $\mathfrak{L}=\mathfrak{L}(s_{-1})$ in the factorization~\eqref{UL} of $J$
with the parameter $s_{-1}=\mathfrak{s}_{0}$ corresponding to the function
\[
{F}(\lambda)=\int_E\frac{d\sigma(t)}{t-\lambda}
\]
is bounded in $\ell^2_{[0,\infty)}$;
\item[(iii)]the sequence $\{{\mathfrak
F}^{[\mathfrak{n}_j/\mathfrak{n}_j]}\}_{j=0}^{\infty}$ of diagonal Pad\'e
approximants converges to ${\mathfrak F}$ locally uniformly in
$\widehat{\dC}\setminus ([-1-\varepsilon, 1+\varepsilon])$
for some $ \varepsilon\ge 0$.
\end{enumerate}
\end{thm}
\begin{proof}
Obviously, the first statement is equivalent to the second one
 by  construction.
So, the equivalence of (ii) and (iii) proves the theorem.
First, suppose that (iii) holds true and, therefore, from Theorem~\ref{PA_Dplus}
we see that~\eqref{Pcond3} is fulfilled. Then, according
to~\eqref{fTh55_1}, \eqref{fTh55_2}, \eqref{fTh55_3}, and~\eqref{fTh55_4},
the operator $L$ is bounded. The implication from (ii) to (iii) follows from
formulas~\eqref{fTh55_2},~\eqref{fTh55_4} and Theorem~\ref{PA_Dplus}.
\end{proof}

In some cases, we can specialize Theorem~\ref{ThOnPD}.

\begin{cor}
Let ${\mathfrak F}$ admit the following representation
\begin{equation}\label{n_kappa_-}
{\mathfrak F}(\lambda)=
\frac{1}{\lambda}\int_{0}^{1}\frac{d\sigma(t)}{t-\lambda}+
\frac{1}{\lambda}\sum_{j=1}^{p}\frac{a_j}{t_j-\lambda}-
\frac{s_{-1}}{\lambda},
\end{equation}
where $\sigma$ is a finite nonnegative measure on $[0,1]$,
$t_j<0$, $a_j\ge 0$ for $j=1,\dots, p$, and $s_{-1}$ is a fixed real number.
Then the following statements hold true:
\begin{enumerate}
\item[(i)]the sequence $\{{\mathfrak
F}^{[\mathfrak{n}_j/\mathfrak{n}_j]}\}_{j=0}^{\infty}$ of diagonal Pad\'e
approximants converges to ${\mathfrak F}$ locally uniformly in
$\widehat{\dC}\setminus ([0,1]\cup\{t_k\}_{k+1}^{p})$;
\item[(ii)]
the polynomials $P_j$ and $Q_j$ corresponding to $\sigma+\sum_{j=1}^{p}a_j\Theta(\cdot-t_j)$
possess the following property
\[
\sup\limits_{j\in\dZ_+}\left|\frac{Q_{\mathfrak{n}_j}(0)-s_{-1}P_{\mathfrak{n}_j}(0)}
{Q_{\mathfrak{n}_j-1}(0)-s_{-1}P_{\mathfrak{n}_j-1}(0)}\right|<\infty;
\]
\item[(iii)]
the operator $\mathfrak{L}$
in the factorization~\eqref{ind_LU} of $\mathfrak{J}$ corresponding to
${\mathfrak F}$
is bounded in $\ell^2_{[0,\infty)}$;
\item[(iv)] the operator $\mathfrak{L}=\mathfrak{L}(s_{-1})$ in the factorization~\eqref{UL} of $J$
with the parameter $s_{-1}=\mathfrak{s}_{0}$ corresponding to the function
\begin{equation}\label{s_kappa_-}
{F}(\lambda)=\int_0^1\frac{d\sigma(t)}{t-\lambda}+\sum_{j=1}^{p}\frac{a_j}{t_j-\lambda}
\end{equation}
is bounded in $\ell^2_{[0,\infty)}$.
\end{enumerate}
\end{cor}
\begin{proof}
Clearly, ${\mathfrak F}\in{\bf D}_{-\infty}^+$.
Actually, statement (i) for the class in question was proved in~\cite{DD07}.
Statement (ii) is an immediate consequence of statement (i) and Theorem~\ref{PA_Dplus}.
The rest follows from Theorem~\ref{ThOnPD}.
\end{proof}
\begin{rem}
1) In fact, in~\cite{DD07} we proved the convergence results for the
Pad\'e approximants
 to
${\bf N}_{\kappa}$-functions having the asymptotic
expansion~\eqref{asymp}.

2) As follows from~\eqref{eq:4.24}, ~\eqref{eq:4.24A} the function
$F$ in~\eqref{s_kappa_-} belongs to the class ${\bf S}^{-\kappa}$,
where $\kappa=p$, if
${F}(0-)=\int_0^1\frac{d\sigma(t)}{t}+\sum_{j=1}^{p}\frac{a_j}{t_j}\le
0$, and $\kappa=p+1$, if $0<{F}(0-)\le \infty$.
\end{rem}

\section{Examples}

\subsection{The appearance of the block structure}
In this subsection we give a family of Jacobi matrices such that the
Christoffel transform of these matrices leads to block tridiagonal
matrices.

Let us consider
a sequence $\{P_j\}_{j=0}^{\infty}$ of monic orthogonal polynomials satisfying the following three term recurrence relation
\[
\lambda P_j(\lambda) = P_{j+1}(\lambda) + c_{j-1}P_{j-1}(\lambda),\quad j\in\dZ_+,
\]
where $c_j>0$ and $P_{-1}(\lambda)=0$, $P_0(\lambda)=1$. For
example, if $c_j\equiv 1/4$, $j\in\dZ_+$, then $P_j$ are Chebyshev polynomials of second kind.
 Clearly, the polynomials $P_j$ are symmetric, i.e.
\begin{equation}\label{SP_property}
P_{j}(-\lambda)=(-1)^{j}P_{j}(\lambda),\quad j\in\dZ_+.
\end{equation}
Besides, the corresponding monic Jacobi matrix has the following
form
\[
J=\begin{pmatrix}
0  & 1 &       &\\
{c}_{0}   & 0    &{1}&\\
        &{c}_1    & 0 &\ddots\\
&       &\ddots &\ddots\\
\end{pmatrix}
\]
and is associated with the Markov function
\[
 F(\lambda)=\int_{-1}^{1}\frac{d\sigma(t)}{t-\lambda},
\]
where $d\sigma$ is a symmetric measure.

It follows from~\eqref{SP_property} that $P_{2j+1}(0)=0$ and $P_{2j}(0)\ne 0$ for every $j\in\dZ_+$.
The latter fact immediately implies that  ${\mathfrak n}_j=2j$, $j\in\dZ_+$, and, thus, ${\mathfrak k}_j=2$ for
$j\in\dZ_+$. Now, Proposition~\ref{prop31} gives the $LU$-factorization  $J=LU$, where
$L$ and $U$ are block lower and upper triangular matrices having the forms
\[
    L=\begin{pmatrix}
I_{\mathfrak{k}_{0}} & 0&       &\\
{L}_{1}   & I_{\mathfrak{k}_{1}}   & 0 &\\
        &L_2    & I_{\mathfrak{k}_{2}} &\ddots\\
&       &\ddots &\ddots\\
\end{pmatrix},\quad
U=\begin{pmatrix}
{U}_{0}   &{D}_{0}&       &\\
{0}   &U_{1}    &{D}_{1}&\\
        &0    &{U}_{2} &\ddots\\
&       &\ddots &\ddots\\
\end{pmatrix},
\]
which entries are as follows
\[
{L}_{j+1}=
\begin{pmatrix}
c_{2j+1}&0\\
0&0\\
\end{pmatrix},\quad
{U}_j= \begin{pmatrix} 0 & 1\\
              c_{2j}& 0 \end{pmatrix},\quad
    {D}_{j}=
\begin{pmatrix}
0&0\\
1&0\\
\end{pmatrix},\quad j\in\dZ_+.
\]
Multiplying $L$ and $U$ the other way around, we arrive at the monic generalized Jacobi matrix
\[
\mathfrak{J}=UL=\begin{pmatrix}
\mathfrak{B}_{0}   &\mathfrak{D}_{0}&       &\\
\mathfrak{C}_{0}   &\mathfrak{B}_1    &\mathfrak{D}_{1}&\\
        &\mathfrak{C}_1    &\mathfrak{B}_{2} &\ddots\\
&       &\ddots &\ddots\\
\end{pmatrix},
\]
where the entries have the following form
\[
{\mathfrak B}_j=\begin{pmatrix} 0 & 1\\
              c_{2j} & 0 \end{pmatrix},\quad
\mathfrak{C}_{j}=
\begin{pmatrix}
0&0\\
{c}_{j+1}c_{j+2}&0\\
\end{pmatrix},\quad
\mathfrak{D}_{j}=
\begin{pmatrix}
0&0\\
1&0\\
\end{pmatrix},\quad j\in\dZ_+.
\]
Finally, observe that $F$ can be represented as follows
\[
 F(\lambda)=\lambda\int_{0}^{1}\frac{d\rho(t)}{t-\lambda^2}\in {\bf N}.
\]
Thus, due to Theorem~\ref{Chris}, we have that $\mathfrak{J}$ is associated with the function
\[
 \lambda F(\lambda)+1=\lambda^2\int_{0}^{1}\frac{d\rho(t)}{t-\lambda^2}
 +1=\int_{0}^{1}\frac{td\rho(t)}{t-\lambda^2}\in {\bf D}_{-\infty}^-.
\]

\subsection{Unboundedness of the Christoffel transform}
Here we give an example of a bounded Jacobi matrix such that its Christoffel transform is an unbounded matrix.

Consider the following monic $2$-periodic Jacobi matrix (see~\cite[Example~1]{DD09})
\[
J=\begin{pmatrix}
a_0       & 1  &       &    &   \\
1       & a_1   & 1   &    &  \\
          & 1   & a_2   &  \ddots \\
 &       & \ddots & \ddots        \\
\end{pmatrix},
\quad a_n=\frac{(-1)^n+1}{2}, \quad n\in\dZ_+,
\]
associated with the function
\[
 F(\lambda)=-\frac{\lambda}{2}+\frac{1}{2}\sqrt{\frac{\lambda(\lambda^2-\lambda-4)}{\lambda-1}},
\]
where the branch is chosen so that $F(\lambda)\to 0$ as $\lambda\to\infty$ (see~\cite{DD09}).

Let us find the $LU$-factorization of $J$, that is, let us represent $J$ as follows
\[
J=LU=\begin{pmatrix}
u_1       & 1   &       &    &   \\
u_1l_1      & u_2+l_1   & 1   &    &  \\
          & u_2 l_2  & u_3+l_2  &  \ddots \\
 &       & \ddots & \ddots        \\
\end{pmatrix}.
\]
Clearly, we have the following relations
\[
u_1=1,\quad u_kl_k=1, \quad u_{2k}+l_{2k-1}=0,\quad u_{2k+1}+l_{2k}=1,\quad k\in\dN.
\]
Next, by induction we have that $u_1=1$, $l_1=1$,  $u_2=-1$, $l_2=-1$, and
\[
 u_{2k+1}=k+1,\quad u_{2k+2}=-\frac{1}{k+1},\quad l_{2k+1}=\frac{1}{k+1},\quad l_{2k+2}=-{(k+1)},\quad k\in\dN.
\]
 Therefore, the operators $L$ and $U$ in $LU$-factorization of $J$ are unbounded.
 Moreover, the Christoffel  transformation
\[
 J_C=UL=
\begin{pmatrix}
u_1+l_1      & 1   &       &    &   \\
u_2l_1      & u_2+l_2   & 1   &    &  \\
          & u_3 l_2  & u_3+l_3  &  \ddots \\
 &       & \ddots & \ddots        \\
\end{pmatrix}=
\begin{pmatrix}
2     & 1   &       &    &   \\
-1      & -2   & 1   &    &  \\
          & -2  & 5/2   &  \ddots \\
 &       & \ddots & \ddots        \\
\end{pmatrix}
\]
is also an unbounded monic generalized Jacobi matrix associated with
the function
\[
\lambda F(\lambda)+1=
-\frac{\lambda^2}{2}+1+\frac{\lambda}{2}\sqrt{\frac{\lambda(\lambda^2-\lambda-4)}{\lambda-1}}
\in{\bf D}_{-\infty}^-.
\]

\end{document}